\documentclass[a4paper]{article}
\usepackage{amsmath,amsthm,amssymb,amscd,graphicx}
\usepackage[all]{xy}
\usepackage{amsfonts}
\usepackage{latexsym}
\usepackage{soul}
\usepackage{comment}
\usepackage{color}

\usepackage{amsmath, latexsym, amssymb}
\numberwithin{equation}{section}
\newtheorem{lemma}{Lemma}[section]
\newtheorem{proposition}[lemma]{Proposition}
\newtheorem{theorem}[lemma]{Theorem}
\newtheorem{corollary}[lemma]{Corollary}

\newtheorem{definition}[lemma]{Definition}
\newtheorem{remark}[lemma]{Remark}

\newcommand{\Aa}{{\mathcal A}}

\newcommand{\Hh}{{\mathcal H}}
\newcommand{\Kk}{{\mathcal K}}
\newcommand{\Mm}{{\mathcal M}}
\newcommand{\Oo}{{\mathcal O}}
\newcommand{\Tt}{{\mathcal T}}
\newcommand{\Uu}{{\mathcal U}}
\newcommand{\Vv}{{\mathcal V}}
\newcommand{\p}{{\partial}}
\renewcommand{\r}{{\rho}}
\newcommand{\C}{{\mathbb C}}
\newcommand{\R}{{\mathbb R}}

\newcommand{\la}{{\langle}}
\newcommand{\ra}{{\rangle}}

\renewcommand{\k}{{\kappa}}

\def\address#1#2{\begingroup
\noindent\parbox[t]{12cm}{%
\small{\scshape\ignorespaces#1}\par\vskip1ex
\noindent\small{\itshape E-mail address}%
\/: #2\par\vskip4ex}\hfill%
\endgroup}%
\makeatother

\title{
Conformal transformations of the pseudo-Riemannian metric 
of a homogeneous pair 
}
\author{Kotaro Kawai
\footnote{2010 Mathematics Subject Classification. 
58B20, 53C29, 58D17, 51F99.}
}
\date{}

\begin{document}

\maketitle

\begin{abstract}
We introduce a new notion of a homogeneous pair 
for a pseudo-Riemannian metric $g$ and a positive function $f$
on a manifold $M$ admitting a free $\R_{>0}$-action. 
There are many examples admitting this structure. 
For example, 
(a) a class of pseudo-Hessian manifolds admitting a free $\R_{>0}$-action 
and a homogeneous potential function 
such as the moduli space of torsion-free $G_2$-structures, 
(b) the space of Riemannian metrics on a compact manifold, 
and (c) many moduli spaces of geometric structures such as torsion-free ${\rm Spin}(7)$-structures 
admit this structure. 
Hence we provide the unified method for the study of these geometric structures.

We consider conformal transformations of 
the pseudo-Riemannian metric $g$ of a homogeneous pair $(g, f)$. 
Showing that the pseudo-Riemannian manifold $(M, (v \circ f) g)$, 
where $v: \R_{>0} \rightarrow \R_{>0}$ is a smooth function, 
has the structure of a warped product,  
we study the geometric structures such as 
the sectional curvature, geodesics and 
the metric completion (if $g$ is positive definite) 
w.r.t. $(v \circ f) g$ 
in terms of those on the level set of $f$. 
In particular, (1) we can generalize the result of Clarke and Rubinstein (\cite{CR2})
about the metric completion of the space of Riemannian metrics w.r.t. 
the conformal transformations of the Ebin metric, 
and (2) two canonical Riemannian metrics on the $G_2$ moduli space 
have different metric completions.

\end{abstract}

\section{Introduction} \label{sec:intro}

In this paper, we introduce a new notion of a homogeneous pair 
for a pseudo-Riemannian metric $g$ and a positive function $f$
on a manifold $M$, possibly infinite dimensional, 
admitting a free $\R_{>0}$-action as follows. 

\begin{definition} \label{def:homog pair intro}
Let $(M,g)$ be a pseudo-Riemannian manifold 
which admits a free $\R_{>0}$-action. 
Let $P \in \mathfrak{X}(M)$ be a vector field generated by the $\R_{>0}$-action. 
Suppose that $f: M \rightarrow \R_{>0}$ is a smooth function and 
$\alpha \in \R - \{ 0 \}$. 

The pair $(g, f)$ is called a 
{\bf homogeneous pair} of degree $\alpha$ if 
\begin{align*}
m_\lambda^* g = \lambda^\alpha g, \qquad
m_\lambda^* f = \lambda^\alpha f, \qquad
g(P, \cdot) = d f, 
\end{align*}
for any $\lambda > 0$, 
where we denote by $m_\lambda$ the action of $\lambda >0$. 
\end{definition}

There are many examples admitting this structure. 
For example, 
a class of pseudo-Hessian manifolds admitting a free $\R_{>0}$-action 
such as the moduli space of torsion-free $G_2$-structures. 
Hessian manifolds appear in many fields of mathematics 
such as information geometry (\cite{AN, AJLS}) 
and the moduli spaces of geometric structures (\cite{Hitchin1, Hitchin2}).
The space of Riemannian metrics on a compact manifold, 
the moduli space of torsion-free ${\rm Spin}(7)$-structures
and many other moduli spaces of geometric structures also admit this structure. 
For more details, see Section \ref{sec:app}. 
Hence we provide the unified method for the study of these geometric structures.

Given a homogeneous pair $(g, f)$, 
we consider the conformal transformations of $g$ 
of the form $(v \circ f) g$, where $v: \R_{>0} \rightarrow \R_{>0}$ is a smooth function. 
There are two reasons to consider this. 

\begin{enumerate}
\renewcommand{\labelenumi}{(\alph{enumi})}
\item
The conformal transformations of $g$ is considered in many examples 
such as the $G_2$ moduli space and the space of Riemannian metrics. 
See Sections \ref{subsec:G2 moduli} and \ref{subsec:Riem met}. 

\item
When $g$ is positive definite and the pseudometric $d_g$ induced from $g$ is a metric
(This is always true when $M$ is finite dimensional. 
In the infinite dimensional case, 
there are examples of a Riemannian metric 
whose induced pseudometric is identically zero (\cite{MM}).), 
the conformal transformation is the simplest way 
to produce the different metric completion w.r.t. the induced metric. 

Clarke and Rubinstein (\cite{CR2}) showed 
that there is an explicit weak Riemannian metric $\tilde g_E$ 
in the conformal class of the Ebin metric $g_E$ 
on the space of Riemannian metrics $\Mm$ 
such that the metric completion induced from $\tilde g_E$ 
is strictly smaller than that from $g_E$. 
They considered this as a first step to remove certain types of degenerations 
so that the canonical functionals such as the curvature, diameter, or injectivity radius 
are controlled by the metric geometry on $\Mm$, 
which is not true for the metric induced from $g_E$ (\cite{Clarke2}). 

We hope that generalizing this by using a homogeneous pair, 
which is done in Theorem \ref{thm:comp}, 
would be useful to remove certain types of degenerations for other geometric problems. 
\end{enumerate}

For the pseudo-Riemannian manifold $(M, (v \circ f) g)$, we first show that 
the following splitting theorem holds as in 
\cite[Theorem 1]{Loftin} and \cite[Lemmas 2.1 and 2.4]{Totaro}.

\begin{theorem} \label{main thm1}
Let $(M,g)$ be a pseudo-Riemannian manifold 
which admits a free $\R_{>0}$-action 
and let $f: M \rightarrow \R_{>0}$ be a smooth function. 
Suppose that $(g, f)$ is a homogeneous pair of degree $\alpha$. 
Then 
\begin{itemize}
\item
for any $l >0$, 
$M_l = \{ x \in M \mid f (x) = l \}$
is a submanifold of $M$ 
and the pullback $g_l$ of $g$ to $M_l$ is a pseudo-Riemannian metric on $M_l$.

\item For a smooth function $v: \R_{>0} \rightarrow \R_{>0}$, 
there is an isometry between 
$\left(\R_{>0} \times M_l, v (r) \left( \frac{1}{\alpha r} dr^2 
+ \frac{r}{l} g_l \right) \right)$
and $(M,(v \circ f) g)$. 
\end{itemize}
\end{theorem}
The more detailed description is given in Theorem \ref{thm:split}. \\

Hence $(M,(v \circ f) g)$ has the structure of a warped product, 
which is a great advantage. 
For example, the geodesic equations get complicated 
under the conformal transformations in general, 
but we can treat them in a simpler way. 
We can say the same for the sectional curvature and the metric completion. 

We summarize results obtained by 
analyzing the sectional curvature, geodesics, and the metric completion 
of the warped product pseudo-Riemannian metric (\ref{eq:def gw}). 
Here, we use the notation of Theorem \ref{main thm1} 
and $(g, f)$ is a homogeneous pair on a manifold $M$. 
\begin{enumerate}
\renewcommand{\labelenumi}{(\arabic{enumi})}
\item 
When $\dim M=2$, we construct a 2 parameter family of 
pseudo-Riemannian metrics of  constant sectional curvature 
in the conformal class of $g$ 
(Corollary \ref{cor:2dim csc Hess}). 
The same is true when $M$ is a direct product of 
such manifolds with $\dim \leq 2$ (Remark \ref{rem:2dim decomp}). 

\item
We construct a 1 parameter family of constant sectional curvature pseudo-Riemannian metrics 
in the conformal class of $g$ 
if the level set $(M_l, g_l)$ has constant sectional curvature (Proposition \ref{prop:conf csc}). 
If the sectional curvature of the level set $(M_l, g_l)$ is bounded 
and $g$ is positive or negative definite, 
we give the bound of the sectional curvature of $(v \circ f) g$ 
for some $v$ (Corollary \ref{cor:bound Hess}). 

For a homogeneous pair $(g, f)$, 
we define a new pseudo-Riemannian metric $\hat g$ in (\ref{eq:def hat g}) 
such that $(\hat g, f)$ is also a homogeneous pair. 
This pseudo-Riemannian metric $\hat g$ 
has a different signature from $g$ and appears in many examples. 
See Sections \ref{sec:Hess} and \ref{sec:app}. 
We give further results of this kind for $\hat g$ 
(Corollaries \ref{cor:conf csc Hessian} and \ref{cor:bound log flat}). 

\item
When $v(r)=r^\beta$ for $\beta \in \R$, 
we describe geodesics of $f^\beta g$ explicitly 
using those in $(M_l, g_l)$ (Proposition \ref{prop:geod explicit}). 
Then we give the conditions on $\beta$ so that 
the function $f$ is convex or concave w.r.t. $f^\beta g$ (Proposition \ref{prop:convex}). 

\item
When $g$ is positive definite and the pseudometric $d_g$ induced from $g$ is a metric, 
we describe the metric completion of $M$ w.r.t. $(v \circ f) g$ for some of $v$ 
in terms of the metric completion of $M_l$ w.r.t. $g_l$
(Theorem \ref{thm:comp}). 
\end{enumerate}

Note that to know the above geometric structures completely, 
we need the information of $(M_l, g_l)$, 
which is obtained from that of $(M,g)$ 
(Lemma \ref{lem:sc flat}, Proposition \ref{prop:comp level}). 
However, by the results above, 
if we have the information of $(M, (v \circ f) g)$ for one $v$, 
we can obtain the information of $(M, (\tilde v \circ f) g)$ for many other $\tilde v$'s.

We can apply results above to many geometric problems. 
See Section \ref{sec:app}. 
We list some particularly important results. 

\begin{enumerate} 
\setcounter{enumi}{4}
\renewcommand{\labelenumi}{(\arabic{enumi})}
\item 
We generalize the result of Clarke and Rubinstein (\cite{CR2})
about the metric completion of the space of Riemannian metrics w.r.t. 
the conformal deformations of the Ebin metric $g_E$ (Theorem \ref{thm:comp riem}).  

They considered the conformal transformations of the form $g_E/f^p$, 
where $f$ is the volume functional and $p \in \mathbb{Z}$. 
We can determine the metric completion w.r.t. $(v \circ f) g_E$ 
for more general functions $v: \R_{>0} \rightarrow \R_{>0}$. 
In particular, we can give infinitely many examples
whose metric completions are strictly smaller than that of $g_E$. 

\item
There are two canonical Riemannian metrics on the $G_2$ moduli space, 
which are related by the conformal transformation. 
Both of them are also studied in detail (cf. \cite{GY, KLL}). 
We can show that they have different metric completions (Corollary \ref{cor:comp G2}). 
\end{enumerate}

This paper is organized as follows. 
In Section \ref{sec:wp}, 
we study in detail the geometric structures of a warped product 
such as the the sectional curvature, the geodesics 
and the metric completion. 
In Section \ref{sec:hp}, 
we prove Theorem \ref{main thm1} (Theorem \ref{thm:split}) 
and results (1)--(4) above by the results in Section \ref{sec:wp}. 
In Section \ref{sec:Hess}, we show that some pseudo-Hessian manifolds admit a homogeneous pair, 
which recovers \cite[Theorem 1]{Loftin} and \cite[Lemmas 2,1 and 2.4]{Totaro}. 
In Section \ref{sec:app}, 
we give examples as previously stated and apply our method. 
In Appendix \ref{app:notation}, 
we summarize the notations and basic definitions used in this paper.

\noindent{{\bf Acknowledgements}}: 
The author would like to thank 
Sergey Grigorian, Spiro Karigiannis and Burt Totaro
for motivating this study. 
He is grateful to Hitoshi Furuhata 
for pointing out that 
the flatness assumption on the connection in 
Proposition \ref{prop:split met} is unnecessary. 
He thanks Takashi Kurose 
for the useful advice on Hessian geometry.
He thanks Sumio Yamada for letting him know 
the theory of the Teichm\"uller space. 
He is indebted to an anonymous referee for the careful reading 
of an earlier version of this paper and useful comments on it. 

This work is supported by 
JSPS KAKENHI Grant Number JP17K14181 
and Research Grants of Yoshishige Abe Memorial Fund.

\section{Warped products} \label{sec:wp}

Let $(X,g_X)$ and $(Y,g_Y)$ be pseudo Riemannian manifolds 
and $\r: X \rightarrow \R_{>0}$ be a positive smooth function on $X$. 
Let $\pi_X: X \times Y \rightarrow X$ and $\pi_Y: X \times Y \rightarrow Y$ 
be the canonical projections. 
The warped product $X \times_\r Y$ is a product manifold $X \times Y$ 
with the pseudo-Riemannian metric $g = \pi_X^* g_X + (\r \circ \pi_X)^2 \pi_Y^* g_Y$:
$$
X \times_\r Y = \left(X \times Y, g=\pi_X^* g_X + (\r \circ \pi_X)^2 \pi_Y^* g_Y \right). 
$$
For simplicity, we drop $\pi_X$ and $\pi_Y$ 
and write $X \times_\r Y = \left(X \times Y, g = g_X + \r^2 g_Y \right).$
We study the geometric structures of warped products in detail 
for the application in Section \ref{sec:hp}.

\subsection{The curvature tensor and the geodesics}

In this subsection, we study the curvature tensor and the geodesics of the warped product 
$X \times_\r Y = \left(X \times Y, g = g_X + \r^2 g_Y \right)$ based on \cite[Section 7]{O'Neill}. 

The vector fields on $X$ and $Y$ are canonically extended to the vector fields on $X \times Y$. 
We identify these vector fields.

\subsubsection{The curvature tensor}

\begin{lemma}[{\cite[Proposition 7.42]{O'Neill}}] \label{lem:wp curv}
Use the notation of Appendix \ref{app:notation}. 
For vector fields 
$x,y,z \in \mathfrak{X} (X)$ and 
$a,b,c \in \mathfrak{X} (Y)$, 
the curvature tensor $R^g$ of $g$ is given as follows. 
\begin{align*}
R^g (x,y)z &= R^{g_X} (x,y) z \ (\in \mathfrak{X} (X)), \\
R^g (a,x)y &= - \frac{ (\nabla^{g_X} d \r) (x,y)}{\r} a,\\
R^g (x,y)a &= R^g (a,w) x=0,\\
R^g (x,a)b &= -\frac{g(a,b)}{\r} \nabla^{g_X}_x ({\rm grad}^{g_X} \r), \\
R^g(a,b)c &= R^{g_Y} (a,b) c + \frac{g_X ({\rm grad}^{g_X} \r, {\rm grad}^{g_X} \r)}{\r^2} (g(a,c)b-g(b,c)a).
\end{align*}
where ${\rm grad}^{g_X} \r \in \mathfrak{X}(X)$ is defined by $g_X ({\rm grad}^{g_X} \r, \cdot) = d\r$. 
\end{lemma}

Note that we adopt the different sign convention of the curvature tensor from \cite[Lemma 3.35]{O'Neill}.

\subsubsection{The geodesics}
Next, we consider the geodesics of the warped product. 
The geodesic equation is described as follows. 

\begin{lemma}[{\cite[Proposition 7.38]{O'Neill}}] \label{lem:wp geod}
Use the notation of Appendix \ref{app:notation}. 
A path $\gamma: J_1 \ni t \mapsto (r(t), y(t)) \in  X \times_\rho Y$, 
where $J_1 \subset \R$ is an open interval, 
is a geodesic 
if and only if 
\begin{align} \label{eq:wp geod 1}
\nabla^{r^* TX}_{\frac{d}{dt}} \dot r &= g_Y(\dot y, \dot y) (\rho \circ r) 
\cdot ({\rm grad}^{g_X} \rho) \circ r, \\
\label{eq:wp geod 2}
\nabla^{y^* TY}_{\frac{d}{dt}} \dot y &= \frac{-2}{\rho \circ r} \frac{d (\rho \circ r)}{dt} \cdot \dot y,
\end{align}
where 
$\nabla^{r^* TX}$ (resp. $\nabla^{y^* TY}$) 
is the induced connection from the Levi-Civita connection of $g_X$ (resp. $g_Y$) 
along the path $t \mapsto r(t)$ (resp. $t \mapsto y(t)$). 
\end{lemma}

Note that (\ref{eq:wp geod 2}) implies that $t \mapsto y(t)$ is a pregeodesic in $Y$. 
That is, a reparametrization of $y$ is a geodesic 
(\cite[Remark 7.39]{O'Neill}).

We rewrite the geodesic equations. We first prove the following.

\begin{lemma} \label{lem:geod repara}
For any path $\hat y:J_2 \rightarrow Y$ and a smooth map $\theta:J_1 \rightarrow J_2$, 
where $J_1, J_2 \subset \R$ are open intervals, 
we have 
$$
\nabla^{(\hat y \circ \theta)^* TY}_{\frac{d}{d t}} \frac{d(\hat y \circ \theta)}{dt}
=
\frac{d^2 \theta}{d t^2} \cdot \left(\frac{d \hat y}{d s} \circ \theta \right) 
+
\left( \frac{d \theta}{d t} \right)^2 \cdot \left( \nabla^{\hat y^* TY}_{\frac{d}{d s}} \frac{d \hat y}{d s} \right) \circ \theta, 
$$
where $\nabla^{(\hat y \circ \theta)^* TY}$ is the induced connection from the Levi-Civita connection 
of $g_Y$ along the path $s \mapsto (\hat y \circ \theta) (s)$. 
\end{lemma}

\begin{proof}
Since 
$
\frac{d(\hat y \circ \theta)}{d t} = \frac{d \theta}{d t} \cdot \left ( \frac{d \hat y}{d s} \circ \theta \right), 
$
we have 
$$
\nabla^{(\hat y \circ \theta)^* TY}_{\frac{d}{d t}} \frac{d(\hat y \circ \theta)}{dt}
=
\frac{d^2 \theta}{d t^2} \cdot \left ( \frac{d \hat y}{d s} \circ \theta \right)
+ 
\frac{d \theta}{d t} \cdot \nabla^{\theta^* \hat y^* TY}_{\frac{d}{dt}} 
\left ( \frac{d \hat y}{d s} \circ \theta \right). 
$$
By the definition of the covariant derivative along the map, we have 
$\nabla^{\theta^* \hat y^* TY}_{\frac{d}{dt}} 
\left ( \frac{d \hat y}{d s} \circ \theta \right)
=
\left( \nabla^{\hat y^* TY}_{\frac{d \theta}{dt}} \frac{d \hat y}{d s} \right) \circ \theta$, 
which gives the proof. 
\end{proof}

Now we rewrite geodesic equations as follows.  

\begin{proposition} \label{prop:geod re}
The geodesic $\gamma: (-\epsilon, \epsilon) \ni t \mapsto (r(t), y(t)) \in X \times_\r Y$ 
with the initial position $(r_0, y_0) \in X \times Y$ and 
the initial velocity $(\dot r_0, \dot y_0) \in T_{r_0} X \times T_{y_0} Y$ is given as follows.

\begin{enumerate}
\renewcommand{\labelenumi}{(\arabic{enumi})}
\item 
The map $r(t)$ is given by the solution of 
\begin{align} \label{eq:wp geod 3}
\nabla^{r^* TX}_{\frac{d}{dt}} \dot r &= \frac{E_1}{(\r \circ r)^3} \cdot ({\rm grad}^{g_X} \rho) \circ r, \\
r(0) &=r_0, \qquad \dot r(0) = \dot r_0, \nonumber
\end{align}
where $E_1=g_Y (\dot y_0, \dot y_0) (\r (r_0))^4$. 
\item
The map $y(t)$ is given by 
\begin{align}\label{eq:wp geod 4}
y(t)=\hat y \left( \int^t_0 \frac{E_3}{(\r \circ r(\tau))^2} d \tau \right),
\end{align}
where $\hat y (s)$ is the geodesic in $(Y, g_Y)$ 
with the initial position $y_0 \in Y$ and 
the initial velocity $\dot y_0 \in T_{y_0} Y$, 
and $E_3 = \r (r_0)^2.$
\end{enumerate}
\end{proposition}

\begin{proof}
It is easy to see that $(r(t),y(t))$ given above satisfies 
$(r(0),y(0))=(r_0,y_0)$ and $(\dot r(0), \dot y (0)) = (\dot r_0, \dot y_0)$. 
We show that $(r(t),y(t))$ satisfies (\ref{eq:wp geod 1}) and (\ref{eq:wp geod 2}). 
Setting 
$$
\theta (t)= \int^t_0 \frac{E_3}{(\r \circ r(\tau))^2} d \tau,
$$
we have $y=\hat y \circ \theta$. 
Then since 
$
\dot y = \frac{dy}{dt}= \frac{E_3}{(\r \circ r)^2} \frac{d \hat y}{ds} \circ \theta,
$
it follows that 
$$
g_Y (\dot y, \dot y) (\r \circ r)^4 = E_3^2 \cdot 
g_Y \left(\frac{d \hat y}{ds} \circ \theta, \frac{d \hat y}{ds}  \circ \theta \right).
$$
Since $\hat y$ is geodesic, 
$g_Y \left(\frac{d \hat y}{ds}  \circ \theta, \frac{d \hat y}{ds} \circ \theta \right)$ 
is constant. 
Thus $g_Y (\dot y, \dot y) (\r \circ r)^4$ is constant, which is equal to $E_1$. 
Then (\ref{eq:wp geod 1}) is immediate from (\ref{eq:wp geod 3}).

Next, we show that $y(t)$ satisfies (\ref{eq:wp geod 2}). 
Lemma \ref{lem:geod repara} implies that 
$$
\nabla^{y^* TY}_{\frac{d}{d t}} \frac{d y}{dt}
=
\frac{d^2 \theta}{d t^2} \cdot \left(\frac{d \hat y}{d s} \circ \theta \right) 
+
\left( \frac{d \theta}{d t} \right)^2 \cdot \left( \nabla^{\hat y^* TY}_{\frac{d}{d s}} \frac{d \hat y}{d s} \right) \circ \theta.
$$
Since $\hat y$ is a geodesic, we have $\nabla^{\hat y^* TY}_{\frac{d}{d s}} \frac{d \hat y}{d s}=0$.
Since 
$
\frac{dy}{dt}= \frac{E_3}{(\r \circ r)^2} \frac{d \hat y}{ds} \circ \theta,
$
we have 
$
\frac{d \hat y}{ds} \circ \theta = 
\frac{(\r \circ r)^2}{E_3} \frac{dy}{dt}.
$
We also compute 
$
\frac{d^2 \theta}{d t^2}= \frac{d}{dt} \left( \frac{E_3}{(\r \circ r)^2} \right)
= \frac{-2 E_3}{(\r \circ r)^3} \frac{d(\r \circ r)}{dt}.
$
Then these equations imply (\ref{eq:wp geod 2}). 
\end{proof}

\subsection{The case $\dim X=1$}

In this subsection, 
we show more detailed descriptions of the curvature tensor and the geodesic equations 
when $X$ is 1-dimensional. 
That is, supposing that $(X,g_X) = (I, \xi d r^2)$, 
where $I \subset \mathbb{R}$ is an open interval, 
$r$ is a coordinate on $I$  
and $\xi= \xi (r)$ is a nowhere vanishing function on $I$, 
we consider the warped product 
$$X \times_\r Y = I \times_\r Y = \left(I \times Y, g = \xi d r^2 + \r^2 g_Y \right). $$

\subsubsection{The curvature tensor}

\begin{lemma} \label{lem:wpsc}
Use the notation of Appendix \ref{app:notation}. 
Set $\p_r = \p/\p r$. We have 
for linearly independent $a,b \in T_y Y$ for $y \in Y$
\begin{align}
g(R^g (\p_r, a)b, \p_r) &= g(a,b) \cdot \frac{- 2 \r'' \xi + \r' \xi'}{2 \r \xi}, 
\label{eq:wpsc1}
\\
K^g (\p_r, a) &= \frac{- 2 \r'' \xi + \r' \xi'}{2 \r \xi^2}, \label{eq:wpsc2}\\
K^g (a, b) &= \frac{1}{\rho^2} \left( K^{g_Y}(a, b) - \frac{(\rho')^2}{\xi} \right). \label{eq:wpsc3}
\end{align}
\end{lemma}

Note that the second equation is independent of $a \in T_y Y$. 
Hence we define a function $K^g(\p_r)$ on $I$ by 
\begin{align} \label{eq:def Kpr}
K^g(\p_r) = K^g (\p_r, a) = \frac{- 2 \r'' \xi + \r' \xi'}{2 \r \xi^2}. 
\end{align}

\begin{proof}
By Lemma \ref{lem:wp curv}, we compute 
\begin{align*}
g (R^g (\p_r,a)b, \p_r)
&=
-\frac{g(a,b)}{\r} \cdot g(\nabla^{g_X}_{\p_r} {\rm grad}^{g_X} \r, \p_r)\\
&=
-\frac{g(a,b)}{\r} \cdot \left( \p_r (d \r (\p_r)) - d \r (\nabla^{g_X}_{\p_r} \p_r) \right)
=
\frac{g(a,b)}{\r} \cdot (- \r'' + d\r (\nabla^{g_X}_{\p_r} \p_r)). 
\end{align*}
By the Koszul formula, it follows that 
\begin{align} \label{eq:LC 1dim}
2 g_X (\nabla^{g_X}_{\p_r} \p_r, \p_r) = \p_r g_X (\p_r, \p_r) = \xi', \qquad
\mbox{and hence} \qquad 
\nabla^{g_X}_{\p_r} \p_r = \frac{\xi'}{2 \xi} \p_r. 
\end{align}
Thus we obtain (\ref{eq:wpsc1}). This also implies (\ref{eq:wpsc2}). 

To prove (\ref{eq:wpsc3}), we use Lemma \ref{lem:wp curv} again and compute 
\begin{align*}
g (R^g (a, b)b, a)
=&
g \left(R^{g_Y} (a, b) b + \frac{g_X ({\rm grad}^{g_X} \r, {\rm grad}^{g_X} \r)}{\r^2} (g(a,b)b-g(b,b)a), a \right)\\
=& 
\rho^2 g_Y (R^{g_Y} (a, b) b, a) + \frac{(\rho')^2}{\rho^2 \xi} \left(g(a,b)^2 - g(a,a) g(b,b) \right),
\end{align*}
where we use ${\rm grad}^{g_X} \r = (\r'/\xi) \p_r$.
Using 
$$
g(a,a) g(b,b) -g(a,b)^2 = \rho^4 (g_Y (a,a) g_Y (b,b) -g_Y (a,b)^2), 
$$
we obtain (\ref{eq:wpsc3}). 
\end{proof}

From these formulae, the sectional curvature for general two vectors in $T(I \times Y)$ 
are computed as follows. 

\begin{lemma} \label{lem:wpgsc}
Take any linearly independent $A=k_1 \p_r + a$ and $B=k_2 \p_r + b$, where 
$k_1, k_2 \in \R$ and $a,b \in T_y Y$ for $y \in Y$. 
If $a$ and $b$ are linearly independent, we have 
\begin{align*} 
K^g(A, B) = \frac{K^g (\p_r) \xi g(k_2 a -k_1 b, k_2 a -k_1 b) + K^g (a,b) \left(g(a,a) g(b,b) - g(a,b)^2 \right)}{\xi g(k_2 a -k_1 b, k_2 a -k_1 b) + g(a,a) g(b,b) - g(a,b)^2}.
\end{align*}
If $a$ and $b$ are linearly dependent, we have 
$K^g(A, B) = K^g (\p_r)$. 
\end{lemma}

Thus $K^g$ is essentially controlled by $K^g (\p_r)$ and $K^g|_{TY \times TY}$.

\begin{proof}
For simplicity, we write 
$R^g (A_1,A_2, A_3, A_4) = g (R^g (A_1, A_2) A_3, A_4)$ 
for $A_1, \cdots, A_4 \in T(I \times Y)$. 
Then 
by Lemma \ref{lem:wp curv}, we have 
$$
R^g(\p_r, a_1, a_2, a_3) = 0
$$
for $a_1, a_2, a_3 \in T_y Y$. 
Using this, we compute 
\begin{align*}
&R^g (A,B,B,A)\\
=& 
R^g (A,B, k_2 \p_r, a) + R^g (A,B, b, k_1 \p_r) + R^g (A,B, b, a)\\
=&
- 2 k_1 k_2 R^g (\p_r, a, b,\p_r)
+ k_2^2 R^g (a, \p_r, \p_r, a)
+ k_1^2 R^g (\p_r, b, b, \p_r)
+ R^g (a,b,b,a).
\end{align*}
By (\ref{eq:wpsc1}), it follows that 
$$
R^g (\p_r, a, b,\p_r) = 
\frac{- 2 \r'' \xi + \r' \xi'}{2 \r \xi} g(a, b)
= K^g (\p_r) \xi g(a, b). 
$$
Then we have 
$$
R^g (A,B,B,A) = 
K^g (\p_r) \xi g(k_2 a -k_1 b, k_2 a -k_1 b) 
+ R^g (a,b,b,a). 
$$
On the other hand, we have
$
g (A,B) = k_1 k_2 \xi + g (a,b).
$
Hence 
\begin{align*}
&g(A,A) g(B,B) -g(A, B)^2 \\
=& (k_1^2 \xi + g (a,a)) \cdot (k_2^2 \xi + g (b,b)) - (k_1 k_2 \xi + g (a,b))^2\\
=&
\xi \left(
- 2 k_1 k_2 g (a,b) + k_2^2 g(a,a) + k_1^2 g(b, b) \right) + g(a,a) g(b,b) - g(a, b)^2\\
=&
\xi g(k_2 a -k_1 b, k_2 a -k_1 b) + g(a,a) g(b,b) - g(a,b)^2.
\end{align*}
Thus since $R^g (a,b,b,a) = K^g (a,b) \left(g(a,a) g(b,b) - g(a,b)^2 \right)$ 
if $a$ and $b$ are linearly independent, 
the proof is done. 
\end{proof}

In this setting, we can characterize the warped product with constant sectional curvature as follows. 
The following statements are obvious from Lemma \ref{lem:wpgsc}. 

\begin{corollary} \label{cor:wpcsc}
For $C \in \R$, $K^g = C$ if and only if 
\begin{align} \label{eq:wpcsc}
K^g (\p_r) = \frac{- 2 \r'' \xi + \r' \xi'}{2 \r \xi^2} = C, \qquad K^g (a,b) = C 
\end{align}
for any linearly independent $a,b \in TY$.
\end{corollary}

\begin{remark}
We see that $K^g =C$ implies that $K^{g_Y}$ is constant. 
In fact, we compute 
\begin{align*}
\frac{d}{d r} \left( C \r^2 + \frac{(\r')^2}{\xi} \right)
=
2C \r \r' + \frac{2 \r' \r''}{\xi} - \frac{(\r')^2 \xi'}{\xi^2}
=
2 \r \r' \left( C - K^g (\p_r) \right)
=0.
\end{align*}
Then (\ref{eq:wpsc3}) implies that $K^{g_Y}$ is constant. 
\end{remark}

Lemma \ref{lem:wpgsc} also yields the following estimates. 
Recall that $\xi g(k_2 a -k_1 b, k_2 a -k_1 b) \geq 0$ and 
$g(a,a) g(b,b) - g(a,b)^2 \geq 0$ 
when $g= \xi (r) dr^2 + \r^2 g_Y$ is definite in the sense of Definition \ref{def:def}.

\begin{corollary} \label{cor:sc bound}
If $g$ is definite in the sense of Definition \ref{def:def}, we have 
$$
\inf_{r \in I, \atop \{a,b\} \in Gr_2(TY)} 
\min \left \{K^g(\p_r), K^g (a,b) \right \}
\leq 
K^g 
\leq
\sup_{r \in I, \atop \{a,b\} \in Gr_2 (TY)}
\max \left \{ K^g(\p_r), K^g (a,b) \right \}, 
$$
where $Gr_2(TY)$ is the  2-Grassmannian bundle over $Y$ 
and $\{a,b\}$ stands for the vector subspace spanned by $a, b \in TY$. 
\end{corollary}

When $Y$ is also 1-dimensional, we can simplify Corollaries \ref{cor:wpcsc} 
and \ref{cor:sc bound}.

\begin{corollary} \label{cor:2dim csc}
Suppose further that $(Y,g_Y)$ is also 1-dimensional. 
Then $K^g = C \in \mathbb{R}$
if and only if 
\begin{align} \label{eq:2dim csc}
K^g (\p_r) = \frac{- 2 \r'' \xi + \r' \xi'}{2 \r \xi^2} = C. 
\end{align}
If $g$ is definite in the sense of Definition \ref{def:def}, we have 
$$
\inf_{r \in I} K^g(\p_r) \leq 
K^g 
\leq
\sup_{r \in I} K^g(\p_r). 
$$
\end{corollary}
Note that the condition (\ref{eq:2dim csc}) is independent of $g_Y$.

\subsubsection{The geodesics}

When $X$ is 1-dimensional, 
(\ref{eq:wp geod 3}) is described more explicitly as follows.

\begin{lemma} \label{lem:wp geod 2}
Use the notation of Proposition \ref{prop:geod re}. 
The equation (\ref{eq:wp geod 3}) holds if and only if 
\begin{align} \label{eq:wp geod 1dim 0}
\ddot r + \frac{\xi' \circ r}{2 \xi \circ r} (\dot r)^2 
-
\frac{E_1}{(\r \circ r)^3} \cdot \frac{\r' \circ r}{\xi \circ r} = 0. 
\end{align}
In particular, we have
\begin{align} \label{eq:wp geod 1dim 1}
(\xi \circ r) \cdot (\dot r)^2 + \frac{E_1}{ (\r \circ r)^2} = E_2, 
\end{align}
where $E_2= \xi (r_0)  (\dot r_0)^2 + g_Y (\dot y_0, \dot y_0) (\r(r_0))^2$. 
\end{lemma}

\begin{proof}
By the identification $\dot r = \dot r \p_r \circ r$, we compute 
\begin{align}
\nabla^{r^* TX}_{\frac{d}{dt}} \dot r 
= 
\ddot r \p_r \circ r + \dot r (\nabla^{g_X}_{\dot r \p_r} \p_r ) \circ r 
\stackrel{(\ref{eq:LC 1dim})}
=
\left( \ddot r + \frac{\xi' \circ r}{2 \xi \circ r} (\dot r)^2 \right) \p_r \circ r.
\end{align}
Then by ${\rm grad}^{g_X} \r = (\r'/\xi) \p_r$, we see that (\ref{eq:wp geod 3}) is 
equivalent to (\ref{eq:wp geod 1dim 0}). 
Multiplying $2 (\xi \circ r) \cdot \dot r$ on both sides of (\ref{eq:wp geod 1dim 0}), 
we have 
$$
2 (\xi \circ r) \cdot \dot r \ddot r + (\xi' \circ r) \cdot (\dot r)^3 
- \frac{2 E_1}{(\r \circ r)^3} \cdot (\r' \circ r) \cdot \dot r = 0. 
$$
Hence 
$$
\frac{d}{dt} \left( (\xi \circ r) \cdot (\dot r)^2 + \frac{E_1}{(\r \circ r)^2} \right) =0,
$$
which gives the proof. 
\end{proof}

To solve (\ref{eq:wp geod 1dim 1}), we can use the method of separation of variables. 
However, it is hard to describe solutions explicitly in general.


\subsection{The special case of the case $\dim X=1$}

In this subsection, we assume that $I= \R_{>0}$ for simplicity and 
the pseudo-Riemannian metric $g$ is of the form (\ref{eq:def gw}).  
This assumption is useful in Section \ref{sec:hp}. 
Assuming this, we can solve many of differential equations in previous subsections explicitly 
and study the sectional curvature, the geodesics and the metric completion in more detail. 

In addition, if we set $k=1$ and $w(r)=1$ in (\ref{eq:def gw}), 
$g=g(w)$ is a cylindrical pseudo-Riemannian metric. 
If we set $k=1$ and $w(r)=r^2$ in (\ref{eq:def gw}), 
$g=g(w)$ is a conical pseudo-Riemannian metric. 
Thus this assumption also provides a framework 
for the unified treatment of these geometrically important examples.

\subsubsection{The sectional curvature}
Assuming that the sectional curvature $K^{g_Y}$ of $(Y, g_Y)$ is constant, 
we construct pseudo-Riemannian metrics of constant sectional curvature. 
We can apply this in Section \ref{subsec:csc met}. 
We begin by the following definition.

\begin{definition}\label{def:w}
Set 
\begin{align*}
\Delta_1 =& \{ (s, C_1, C_2) \in \R^3 \mid s>0, C_1>0, \ C_2 \in \R \}, \\
\Delta_2 =& \{ (s, C_1, C_2) \in \R^3 \mid s=0, C_1 \geq 0, \ C_2 \in \R \}, \\
\Delta_3 =& \{ (s, C_1, C_2) \in \R^3 \mid s<0, C_1 \leq 0, C_2 \in \R \}.
\end{align*}
Then define a function $w(s, C_1, C_2, r)$ for $(s, C_1, C_2) \in \Delta_1 \cup \Delta_2 \cup \Delta_3$ 
and (generic) $r \in \R_{>0}$ as follows. 
\begin{itemize}
\item
For $(s, C_1, C_2) \in \Delta_1$, set 
$$
w(s, C_1, C_2, r) = 
\frac{C_1}{s \cdot \left(\cosh (\sqrt{C_1} (\log r + C_2)) \right)^2}. 
$$
\item
For $(0, C_1, C_2) \in \Delta_2$, set 
$$
w(0, C_1, C_2, r) = e^{C_2} r^{\pm 2 \sqrt{C_1}}.
$$
\item
For $(s, C_1, C_2) \in \Delta_3$ with $C_1 <0$, set 
$$
w(s, C_1, C_2, r) = 
\frac{C_1}{s \cdot \left( \sin (\sqrt{- C_1} (\log r + C_2)) \right)^2}
$$
For $(s, 0, C_2) \in \Delta_3$, set 
$$
w(s, 0, C_2, r) = 
\frac{-1}{s \cdot (\log r + C_2)^2}
\left( = \lim_{C_1 \rightarrow - 0} w(s, C_1, C_2, r)
\right). 
$$
\end{itemize}
\end{definition}

\begin{proposition} \label{prop:tech}
Let $(Y, g_Y)$ be a pseudo Riemannian manifold. 
Fix $k \in \R - \{ 0 \}$. For  
a smooth function $w: \R_{>0} \rightarrow \R_{>0}$, 
define a pseudo-Riemannian metric $g=g(w)$ on $\R_{>0} \times Y$ by 
\begin{align} \label{eq:def gw}
g=g(w)= \frac{k w(r)}{r^2} d r^2+ w(r) g_Y. 
\end{align}
Defining  
functions $\xi, \r : \R_{>0} \rightarrow \R_{>0}$ by 
$$
\xi (r) = \frac{k w(r)}{r^2}, \qquad
\r (r) =  \sqrt{w(r)},  
$$
we have the following. 
\begin{enumerate}
\renewcommand{\labelenumi}{(\arabic{enumi})}
\item 
Recall (\ref{eq:def Kpr}). 
Given $C \in \R$, the differential equation 
$$
K^{g(w)} (\p_r) = \frac{-2 \r'' \xi + \r' \xi'}{2 \r \xi^2} = C 
$$
w.r.t. $w(r)$ 
has a 2 parameter family of solutions given by 
$w(r) = w (k C, C_1, C_2, r)$ for 
$(C_1, C_2) \in \R^2$ such that $(k C, C_1, C_2) \in \Delta_1 \cup \Delta_2 \cup \Delta_3$, 
where we use the notation in Definition \ref{def:w}.

\item 
For $g= g(w (k C, C_1, C_2, \cdot))$, we have 
\begin{align} \label{eq:sec ab}
K^g (a,b) = \frac{1}{w(r)} \left( K^{g_Y} (a,b) - \frac{C_1}{k} \right) + C
\end{align}
for linearly independent $a,b \in TY.$

\item 
The pseudo-Riemannian metric $g= g(w (k C, C_1, C_2, \cdot))$ has constant sectional curvature $C$ 
if and only if $g_Y$ has constant sectional curvature $C_1/k$. 
\end{enumerate}
\end{proposition}

\begin{remark}
By fixing $(k C, C_1, C_2)$, 
the function $w (k C, C_1, C_2, r)$ of $r$ is defined for all $r>0$ when $k C \geq 0$. 
When $k C < 0$, it is only defined 
on the complement of the discrete set of $\R_{>0}$. 
\end{remark}

\begin{proof}
Setting $w(r) = e^{2 W(r)}$ for a smooth function $W:\R_{>0} \rightarrow \R$, we have
$$
\xi = \frac{k e^{2 W}}{r^2}, \qquad \r = e^W.
$$
Then 
\begin{align*}
-2 \r'' \xi + \r' \xi'
=&
k \left( -2(W''+(W')^2) e^W \cdot \frac{e^{2 W}}{r^2} 
+ W' e^W \left( \frac{2 W' e^{2 W}}{r^2} - \frac{2 e^{2 W}}{r^3} \right) \right)\\
=&
k e^{3 W} \left(- \frac{2 W''}{r^2} - \frac{2 W'}{r^3} \right)\\
=&
\frac{-2 k e^{3 W}}{r^3} \frac{d}{dr} (r W').
\end{align*}
Since $2 \r \xi^2 = 2 k^2 e^{5 W}/r^4$, we obtain 
\begin{align*} 
K^{g(w)} (\p_r) = 
\frac{-2 \r'' \xi + \r' \xi'}{2 \r \xi^2} = 
- \frac{r }{k e^{2 W}} \frac{d}{dr} (r W'). 
\end{align*}
Thus $K^{g(w)} (\p_r) =C$ is equivalent to  
$$
r \frac{d}{dr} (r W') = - k C e^{2 W}.
$$
Multiplying $W'$ on both sides, we have 
$$
\frac{d}{dr} ((r W')^2) = -k C \frac{d}{d r} e^{2 W}.
$$
Thus we obtain 
\begin{align} \label{eq:ode}
(r W')^2 = -k C e^{2 W} + C_1 
\end{align}
for $C_1 \in \R$. 
This can be solved by the method of separation of variables. 
After a straightforward computation, we obtain the following.

\begin{itemize}
\item When $k C >0$, we have 
$$
W(r)=- \log \left( \sqrt{\frac{k C}{C_1}} \cosh (\sqrt{C_1} (\log r + C_2)) \right)
\qquad (C_1>0, \ C_2 \in \R). 
$$
\item When $C =0$, we have 
$$
W(r)= \pm \sqrt{C_1} \cdot \log r + \frac{C_2}{2}
\qquad (C_1 \geq 0, \ C_2 \in \R), 
$$
\item When $k C <0$, we have 
$$
W(r) = -\log \left| \sqrt{-k C} (\log r + C_2) \right| \qquad (C_2 \in \R),
$$
which corresponds to $C_1=0$, 
or 
$$
W(r) = - \log \left| \sqrt{\frac{k C}{C_1}} \sin (\sqrt{- C_1} (\log r + C_2)) \right|  \qquad (C_1<0, \ C_2 \in \R).  
$$
\end{itemize}
Then we obtain (1) via $w (r) = e^{2 W(r)}$. 
For the proof of (2), 
recall by (\ref{eq:wpsc3}) that 
$K^g (a, b) 
= \frac{1}{\rho^2} \left( K^{g_Y}(a, b) - \frac{(\rho')^2}{\xi} \right).
$
Then we compute 
$$
- \frac{(\r')^2}{\xi} 
= - \frac{(W')^2 e^{2 W} \cdot r^2}{k e^{2 W}}
\stackrel{(\ref{eq:ode})}
= 
\frac{1}{k} (- C_1+ k C e^{2 W}) 
=
- \frac{C_1}{k} + C \r^2, 
$$
which gives the proof of (2). 
The statement (3) is immediate from (1), (2) and Corollary \ref{cor:wpcsc}. 
\end{proof}

\begin{remark} \label{rem:symm gw}
For a function $w_1: \R_{>0} \rightarrow \R_{>0}$, 
define a function $w_2: \R_{>0} \rightarrow \R_{>0}$ by 
$w_2 (r) =w_1 (1/r)$. Then 
$(\R_{>0} \times Y, g(w_1))$ and $(\R_{>0} \times Y, g(w_2))$ are 
isometric via $(r,y) \mapsto (1/r, y)$. 

This is because $dr^2/r^2 = (d \log r)^2$ is invariant under $r \mapsto 1/r$. 
In particular, the space of solutions $w$ of $K^{g(w)} (\p_r) = C$ given in 
Proposition \ref{prop:tech} (1) is invariant under $w(r) \mapsto w(1/r)$. 
\end{remark}

We have the following sectional curvature bound 
by Corollary \ref{cor:sc bound} and Proposition \ref{prop:tech}.

\begin{corollary} \label{cor:sc bound wp}
Use the notation of Definition \ref{def:w} and Proposition \ref{prop:tech}. 
Suppose that $g = g(w(k C, C_1, C_2, \cdot))$, where $(k C, C_1, C_2) \in \Delta_1 \cup \Delta_2 \cup \Delta_3$, 
is definite in the sense of Definition \ref{def:def}. Then we have 
\begin{itemize}
\item
$K^g \geq C$ when $K^{g_Y} \geq C_1/k$, and 
\item
 $K^g \leq C$ when $K^{g_Y} \leq C_1/k$. 
\end{itemize}
Furthermore, $K^g = C$ if and only if $K^{g_Y}=C_1/k$. 
\end{corollary}

When $\dim Y =1$, we do not need the 
assumption on $K^{g_Y}$ by Corollary \ref{cor:2dim csc}. 
Then Proposition \ref{prop:tech} (1) implies the following.

\begin{corollary} \label{cor:2dim csc wp}
Use the notation of Definition \ref{def:w} and Proposition \ref{prop:tech}. 
In addition to the assumptions of Proposition \ref{prop:tech}, 
suppose further that $\dim Y = 1$.  
Then given $C \in \R$, 
$g=g(w)$ has constant sectional curvature $C$ if 
$w=w(k C, C_1, C_2, \cdot)$, 
where $C_1, C_2 \in \R$ such that $(k C, C_1, C_2) \in \Delta_1 \cup \Delta_2 \cup \Delta_3$. 
\end{corollary}

\subsubsection{The geodesics}
By Proposition \ref{prop:geod re} and Lemma \ref{lem:wp geod 2}, 
we now describe the geodesics explicitly 
for $g(w(0,C_1,C_2, \cdot))$ for $C_1 \geq 0$ and $C_2 \in \R$. 
(We tried to describe the geodesics explicitly for 
$g(w(s,C_1,C_2, \cdot))$ for any $s$, but we could do it only when $s=0$. )

Setting $w(r)= r^{C_0}$ for $C_0 \in \R$, 
we consider the geodesics for the pseudo-Riemannian metric 
\begin{align}\label{eq:geod met}
g=g(r^{C_0})= k r^{C_0-2} dr^2+ r^{C_0} g_Y. 
\end{align}
Note that since $g(\lambda w) = \lambda g(w)$ for $\lambda > 0$ 
and $w:\R_{>0} \rightarrow \R_{>0}$, 
and the Levi-Civita connection is invariant under the scalar multiplication of a pseudo-Riemannian metric, 
we may assume that the coefficient of $r^{C_0}$ is $1$.

\begin{proposition}\label{prop:geod wp explicit}
Use the notation of Definition \ref{def:w} and Proposition \ref{prop:tech}. 
The geodesic $\gamma: (-\epsilon, \epsilon) \ni t \mapsto (r(t), y(t)) \in \R_{>0} \times Y$ 
with the initial position $(r_0, y_0) \in \R_{>0} \times Y$ and 
the initial velocity $(\dot r_0, \dot y_0) \in \R \times T_{y_0} Y$ 
w.r.t. $g=g(r^{C_0})$ is given as follows. 

\begin{enumerate}
\renewcommand{\labelenumi}{(\arabic{enumi})}
\item 
When $C_0 \neq 0$, 
\begin{align}\label{eq:geod wp}
\begin{split}
r(t) &= r_0 \left(1 + C_0 \frac{\dot r_0}{r_0} t + C_0^2 F t^2 \right)^{\frac{1}{C_0}}, \\
y(t) &= \hat y \left( \int^t_0 \frac{d \tau}{1 + C_0 \frac{\dot r_0}{r_0} \tau + C_0^2 F \tau^2} \right), 
\end{split}
\end{align}
where $F= \frac{1}{4} \left(\left( \frac{\dot r_0}{r_0} \right)^2 + \frac{g_Y (\dot y_0, \dot y_0)}{k} \right)$
and $\hat y (s)$ is the geodesic in $(Y, g_Y)$ 
with the initial position $y_0 \in Y$ and the initial velocity $\dot y_0 \in T_{y_0} Y$. 
\item
When $C_0=0$, 
\begin{align}\label{eq:geod wp 0}
r(t) = r_0 e^{\frac{\dot r_0}{r_0} t}, \qquad
y(t)= \hat y (t). 
\end{align}
\end{enumerate}
\end{proposition}

\begin{remark} \label{rem:explicit geod}
By a straightforward computation, 
the integral $\int^t_0 \frac{d \tau}{1 + C_0 \frac{\dot r_0}{r_0} \tau + C_0^2 F \tau^2}$ 
in (\ref{eq:geod wp}) can be explicitly computed as follows. 
$$
\left\{ \begin{array}{ll}
\frac{2}{C_0} \sqrt{\frac{k}{g_Y (\dot y_0, \dot y_0)}} \arctan 
\left ( \frac{\sqrt{\frac{g_Y (\dot y_0, \dot y_0)}{k}} t}{\frac{\dot r_0}{r_0} t + \frac{2}{C_0}} \right )
 & \mbox{if} \quad F \neq 0, \ k g_Y (\dot y_0, \dot y_0) >0, \\
\frac{4 r_0^2 F t}{\dot r_0 (2 r_0 C_0 F t + \dot r_0)} 
& \mbox{if} \quad F \neq 0, \ k g_Y (\dot y_0, \dot y_0) = 0, \\
\frac{2}{C_0} \sqrt{\frac{-k}{g_Y (\dot y_0, \dot y_0)}} {\rm arctanh}  
\left ( \frac{\sqrt{\frac{- g_Y (\dot y_0, \dot y_0)}{k}} t}{\frac{\dot r_0}{r_0} t + \frac{2}{C_0}} \right )
 & \mbox{if} \quad F \neq 0, \ k g_Y (\dot y_0, \dot y_0) < 0, \\
\frac{r_0}{C_0 \dot r_0} \log \left( 1+ C_0 \frac{\dot r_0}{r_0} t \right)
& \mbox{if} \quad F = 0, \ \dot r_0 \neq 0, \\
t & \mbox{if} \quad F = 0, \ \dot r_0 = 0. \\
\end{array} \right.
$$
These formulae implies that geodesics are not defined for all $t \in \R$ in general. 
In particular, $g(r^{C_0})$ in (\ref{eq:geod met}) is incomplete if $C_0 \neq 0$. 
It is complete 
if $C_0=0$ and $g_Y$ is complete.  
This is consistent with Theorem \ref{thm:comp wp}. 
\end{remark}

\begin{proof}
Setting  
$$
\xi (r) =k r^{C_0-2}, \qquad 
\r (r) = r^{C_0/2}, 
$$
we have to solve (\ref{eq:wp geod 1dim 0}) and compute (\ref{eq:wp geod 4})
by Proposition \ref{prop:geod re} and Lemma \ref{lem:wp geod 2}. 
By Lemma \ref{lem:wp geod 2}, 
we first solve (\ref{eq:wp geod 1dim 1}). 
It is equivalent to 
\begin{align}\label{eq:geod wp ori}
k r^{2 C_0-2} (\dot r)^2 + E_1 = E_2 r^{C_0}. 
\end{align}

Now suppose that $C_0=0$. Then (\ref{eq:geod wp ori}) implies that 
$r^{-2} (\dot r)^2$ is constant, 
and hence $r(t)=L_1 e^{t L_2}$ for $L_1, L_2 \in \R$. 
Since $r(0)=r_0$ and $\dot r (0)= \dot r_0$, we have 
$r(t)=r_0 e^{t \dot r_0/r_0}$. 
It is straightforward to see that this satisfies (\ref{eq:wp geod 1dim 0}).  
Since $(\r (r))^2 = 1$, (\ref{eq:wp geod 4}) implies that $y(t)=\hat y(t)$. \\

Next, suppose that $C_0 \neq 0$. 
Setting $s(t)=(r(t))^{C_0}$, we have 
$\dot s = C_0 r^{C_0-1} \dot r$. 
Then (\ref{eq:geod wp ori}) becomes 
$$
\frac{k}{C_0^2} (\dot s)^2 + E_1 =E_2 s,
$$
Differentiating this equation, we have 
$$
\dot s \left(\frac{2 k}{C_0^2}  \ddot s - E_2 \right)=0,
$$
which implies that 
$s(t) = F_0 + F_1 t + \frac{C_0^2 E_2}{4 k} t^2$ for $F_0, F_1 \in \R$ or $s(t)$ is constant. 
When 
$s(t) = F_0 + F_1 t + \frac{C_0^2 E_2}{4 k} t^2$, 
since 
$r(0)=r_0$ and $\dot r(0) = \dot r_0$, 
we have 
$F_0=r_0^{C_0}$ and 
$F_1 = C_0 r_0^{C_0-1} \dot r_0$. 
Since 
\begin{align*}
\frac{C_0^2 E_2}{4 k} 
= \frac{C_0^2}{4 k} \left( \xi (r_0)  (\dot r_0)^2 + g_Y (\dot y_0, \dot y_0) (\r(r_0))^2 \right)
= C_0^2 r_0^{C_0} \cdot \frac{1}{4} \left( \left( \frac{\dot r_0}{r_0} \right)^2 
+ \frac{g_Y (\dot y_0, \dot y_0)}{k} \right), 
\end{align*}
we see that $r(t)=(s(t))^{1/C_0}$ is given by the first equation of (\ref{eq:geod wp}). 
It is straightforward to see that this satisfies (\ref{eq:wp geod 1dim 0}). 

When $s(t)$ is constant, and hence $r(t)$ is constant, 
(\ref{eq:wp geod 1dim 0}) implies that $E_1=0$. 
Then (\ref{eq:geod wp ori}) implies that $E_2=0$. 
By the definitions of $E_1$ and $E_2$ in 
Proposition \ref{prop:geod re} and Lemma \ref{lem:wp geod 2}, 
it follows that $g(\dot y_0, \dot y_0) =0$ and $\dot r_0=0$. 
Hence this case is reduced to (\ref{eq:geod wp}). 

Since 
$(\r \circ r (t))^2 = r(t)^{C_0} = s(t)$, 
(\ref{eq:wp geod 4}) implies the second equation of (\ref{eq:geod wp}). 
\end{proof}

The next corollary is used to prove Proposition \ref{prop:convex}. 

\begin{corollary}\label{cor:convex wp}
Let 
$\gamma: (-\epsilon, \epsilon) \ni t \mapsto (r(t), y(t)) \in \R_{>0} \times Y$ 
be a geodesic 
w.r.t. the pseudo-Riemannian metric $g(r^{C_0})$ in (\ref{eq:geod met}). 

\begin{enumerate}
\renewcommand{\labelenumi}{(\arabic{enumi})}
\item
The function
$r(t)$ is a convex function if one of the following conditions holds. 
\begin{itemize}
\item $C_0=0$, 
\item $0 < C_0 \leq 2$ and $k g_Y$ is positive definite, 
\item $C_0 < 0$ and $k g_Y$ is negative definite, 
\end{itemize}

\item
The function $r(t)$ is a concave function if 
$C_0 \geq 2$ and $k g_Y$ is negative definite. 
\end{enumerate}
\end{corollary}

\begin{proof}
By Proposition \ref{prop:geod wp explicit}, 
it is obvious that $r(t)$ is convex when $C_0=0$. 
Suppose that $C_0 \neq 0$. 
A straightforward calculation gives that 
$$
\frac{d^2 r(t)}{d t^2} 
=
\frac{r_0}{C_0} \cdot 
\left(1 + C_0 \frac{\dot r_0}{r_0} t + C_0^2 F t^2 \right)^{\frac{1}{C_0}-2} p(t),
$$
where $p(t)$ is a polynomial given by 
$$
p(t) = 
(4-2 C_0) \left( C_0^3 F^2 t^2 + C_0^2 \frac{\dot r_0}{r_0} F t \right)
+ 
\left(
(1-C_0) C_0 \left( \frac{\dot r_0}{r_0} \right)^2 + 2 C_0^2 F \right). 
$$
When $C_0=2$, we have 
$$
p(t) 
= -2 \left( \frac{\dot r_0}{r_0} \right)^2 + 8 F
= \frac{2 g_Y (\dot y_0, \dot y_0)}{k}, 
$$
which gives the statement for $C_0=2$. 

Suppose that $C_0 \neq 0, 2$. 
It is also straightforward to see that 
the discriminant ${\rm disc}(p(t))$ of the quadratic $p(t)$ is given by 
$$
{\rm disc}(p(t)) = 4 (C_0 -2) C_0^3 F^2 \cdot \frac{g_Y (\dot y_0, \dot y_0)}{k}. 
$$
Since 
$r(t)$ is convex (resp. concave) 
if $4-2 C_0 > 0$ (resp. $4-2 C_0 < 0$) and ${\rm disc}(p(t)) \leq 0$, 
we obtain the statement.  
\end{proof}

\begin{remark} \label{rem:kgY g gY}
Use the notation of Definition \ref{def:def}. 
By the definition of $g=g(r^{C_0})$ in (\ref{eq:geod met}), we can rephrase the conditions 
in Corollary \ref{cor:convex wp} as follows. 
\begin{itemize}
\item
The pseudo-Riemannian metric 
$k g_Y$ is positive definite if and only if $g$ and $g_Y$ are definite. 
\item
The pseudo-Riemannian metric 
$k g_Y$ is negative definite if and only if $g$ is Lorentzian and $g_Y$ is definite. 
\end{itemize}
\end{remark}

\subsubsection{The metric completion}

In this subsection, 
we consider the pseudo-Riemannian metric $g=g(w)$ given in (\ref{eq:def gw}) again. 
We assume the following. 
\begin{itemize}
\item
The pseudo-Riemannian metric $g=g(w)$ given in (\ref{eq:def gw}) is positive definite. 
That is, $k>0$ and $g_Y$ is positive definite. 
\item
The pseudometric $d_g$ induced from $g=g(w)$
is a metric.
(This is always true when $Y$ is finite dimensional. 
In the infinite dimensional case, 
there are examples of a Riemannian metric 
whose induced pseudometric is identically zero (\cite{MM}). 
Note that Lemmas \ref{lem:unif conti} and \ref{lem:R12} (1)
imply that $d_g$ is a metric if 
the pseudometric $d_{g_Y}$ induced from $g_Y$ is a metric. 
)
\end{itemize}

We study the metric completion of $\R_{>0} \times Y$ w.r.t. $d_g$ following \cite[Section 5]{CR2}. 
Recall that the metric $d_g$ between $(r_0, y_0)$ and $(r_1, y_1) \in \R_{>0} \times Y$ 
is given by
\begin{align*}
&d_g ((r_0, y_0), (r_1, y_1)) \\
=& 
\inf \left\{ L_g (c)  \ \middle| \ 
\begin{array}{l}
c = (r \circ c, y \circ c): [0,1] \rightarrow \R_{>0} \times Y \mbox{ is a piecewise smooth path} \\
\mbox{with } c(0) = (r_0, y_0) \mbox{ and } c(1)=(r_1, y_1) \\
\end{array} 
\right \}, 
\end{align*}
where 
\begin{align} \label{eq:def length}
\begin{split}
L_g (c) 
=& \int_0^1 \left| \frac{d c}{d t} (t) \right|_g dt\\
=& \int_0^1 \sqrt{\frac{k w ((r \circ c) (t))}{((r \circ c) (t))^2}   \left( \frac{d (r \circ c)}{dt} (t) \right)^2
+ w ((r \circ c) (t)) \left| \frac{d (y \circ c)}{d t} (t) \right|_{g_Y}^2 } dt. 
\end{split}
\end{align}
Here, we use the notation of Appendix \ref{app:notation}. 
Similarly, we can define the metric $d_{g_Y}$ induced from $g_Y$. \\

To study the metric completion, we first prove the following lemmas. 
Fixing $R_0 \in \R_{>0}$, define a strictly increasing function $T:\R_{>0} \rightarrow \R$ by 
\begin{align} \label{eq:def T}
T(r)=\int^r_{R_0} \frac{\sqrt{k w(q)}}{q} dq. 
\end{align}

\begin{lemma} \label{lem:unif conti}
For any $(r_0, y_0), (r_1, y_1) \in \R_{>0} \times Y$, we have 
$$
d_g ((r_0, y_0), (r_1, y_1)) \geq |T(r_1)-T(r_0)|. 
$$
In particular, 
$\R_{>0} \times Y \ni (r,y) \mapsto r \in \R_{>0}$ is continuous w.r.t. $d_g$. 
\end{lemma}

\begin{proof}
Let $c = (r \circ c, y \circ c): [0,1] \rightarrow \R_{>0} \times Y$ be a piecewise smooth path 
with $c(0) = (r_0, y_0)$ and $c(1)=(r_1, y_1)$.
By (\ref{eq:def length}), we compute 
$$
L_g(c) \geq 
\int_0^1 \frac{\sqrt{k w (r \circ c)}}{r \circ c} 
\left|\frac{d}{dt} (r \circ c) \right| dt
\geq
\left|
\int_0^1  
\frac{\sqrt{k w (r \circ c)}}{r \circ c} 
\frac{d}{dt} (r \circ c) dt 
\right|
= 
|T(r_1) - T(r_0)|. 
$$
\end{proof}

As $T$ is strictly increasing, 
it converges in $\R \cup \{ - \infty \}$ 
(resp. $\R \cup \{ \infty \}$) as $r \rightarrow 0$ (resp. $r \rightarrow \infty$). 
Set 
$$
T_0 = \lim_{r \rightarrow 0} T(r) \in \R \cup \{ - \infty \}, \qquad
T_\infty = \lim_{r \rightarrow \infty} T(r) \in \R \cup \{ \infty \}. 
$$
Then the following is immediate from Lemma \ref{lem:unif conti}. 
This is useful to study the metric completion w.r.t. $d_g$.

\begin{corollary} \label{cor:conv r}
If $\{ (r_k, y_k) \} \subset \R_{>0} \times Y$ is a $d_g$-Cauchy sequence, 
$\{ r_k \}$ converges in 
$$
\left\{ \begin{array}{lll}
\R_{> 0}                              & \mbox{if} \quad T_0= - \infty, \  &T_\infty = \infty, \\
\{ 0 \} \cup \R_{> 0}                         & \mbox{if} \quad T_0 \in \R, \  &T_\infty = \infty, \\
\R_{> 0} \cup \{ \infty \}      & \mbox{if} \quad T_0= - \infty, \  &T_\infty \in \R, \\
\{ 0 \} \cup \R_{> 0} \cup \{ \infty \} & \mbox{if} \quad T_0 \in \R, \  &T_\infty \in \R. \\
\end{array} \right.
$$
\end{corollary}

\begin{lemma} \label{lem:R12}
Fix $0<R_1 <R_2$. There exist 
$\delta =\delta (R_1, R_2, T)$, a constant depending on $R_1,R_2$ and $T$, 
and  
$C' =C' (R_1, R_2, w), C''=C'' (R_1, R_2, w) > 0$, 
constants depending on $R_1,R_2$ and $w$, 
such that for any 
$(r_0,y_0), (r_1,y_1) \in (R_1, R_2) \times Y$
\begin{enumerate}
\renewcommand{\labelenumi}{(\arabic{enumi})}
\item 
$d_g \left((r_0,y_0), (r_1,y_1) \right) < \delta \Rightarrow 
d_{g_Y} (y_0,y_1) \leq C' d_g \left( (r_0,y_0), (r_1,y_1) \right). 
$
\item
$
d_g ((r_0,y_0), (r_1,y_1)) \leq 
\left| T(r_0) - T(r_1) \right| + C'' d_{g_Y} (y, y'). 
$
\end{enumerate}
\end{lemma}

\begin{proof}
First, we show that 
there exists $\delta = \delta (R_1, R_2, T) > 0$  
such that for any 
$(r,y) \in (R_1, R_2) \times Y$ and $(r',y') \in \R_{>0} \times Y$ 
\begin{align}\label{eq:unif conti cor}
d_g \left( (r,y), (r',y') \right) < 2 \delta \Rightarrow \frac{R_1}{2} < r' < 2 R_2. 
\end{align}

By Lemma \ref{lem:unif conti}, the map 
$\R_{>0} \times Y \ni (r,y) \mapsto T(r) \in \R$ is uniformly continuous w.r.t. $d_g$. 
Then for $\epsilon = 
{\rm min} \left\{ T(R_1) - T \left(R_1/2 \right), T(2 R_2) - T(R_2) 
\right\} >0$, 
there exists $\delta = \delta (R_1, R_2, T) >0$ such that 
for any $(r,y), (r',y') \in \R_{>0} \times Y$
$$
d_g \left( (r,y), (r',y') \right) < 2 \delta \Rightarrow |T(r) - T(r')| < \epsilon. 
$$
In particular, if $(r,y) \in (R_1,R_2) \times Y$, we see that 
$$
T \left( \frac{R_1}{2} \right)
\leq T(R_1) - \epsilon 
< T(r) - \epsilon 
< T(r')
< T(r) + \epsilon
< T(R_2) + \epsilon
\leq T(2R_2), 
$$
hence we obtain (\ref{eq:unif conti cor}). \\

Now we prove (1). Suppose that $d_g \left( (r_0,y_0), (r_1,y_1) \right) < \delta$ for $\delta$ given above. 
For any $0 < \epsilon < \delta$, take a piecewise smooth path 
$\{ c(t) \}_{t \in [0,1]}$ 
connecting $(r_0,y_0)$ and $(r_1,y_1)$ 
such that 
$
L_g (c) < d_g \left( (r_0,y_0), (r_1,y_1) \right)  + \epsilon. 
$
Then for any $t \in [0,1]$, we have 
$$
d_g ((r_0,y_0), c(t)) \leq L_g (c|_{[0,t]}) \leq L_g (c) < \delta + \epsilon < 2 \delta. 
$$
Hence by (\ref{eq:unif conti cor}), it follows that 
$\frac{R_1}{2} < (r  \circ c) (t) < 2 R_2$ for any $t \in [0,1]$. 
Thus 
setting 
$1/C' = {\rm min} \left \{ \sqrt{w(r)} \mid r \in \left[ \frac{R_1}{2}, 2 R_2 \right] \right \},$
we obtain by (\ref{eq:def length})
$$
L_g (c) 
\geq \int_0^1 \sqrt{w ((r \circ c) (t))} \left| \frac{d (y \circ c)}{d t} (t) \right|_{g_Y} dt
\geq  \frac{L_{g_Y} (y \circ c)}{C'} 
\geq \frac{d_{g_Y} (y_0,y_1)}{C'}. 
$$
Since $L_g (c) < d_g \left( (r_0,y_0), (r_1,y_1) \right)  + \epsilon$ and $\epsilon$ 
is arbitrarily small, we obtain (1). \\

Next, we prove (2). 
Define a path $c:[0,1] \rightarrow \R_{>0} \times Y$ by 
$c(t)= \left( (r_1-r_0)t+ r_0, \tilde y (t) \right)$, 
where $\tilde y: [0,1] \rightarrow Y$ is a path such that 
$\tilde y (0)=y_0$ and $\tilde y (1)=y_1$. Then by (\ref{eq:def length}), we see that 
$$
d_g \left( (r_0,y_0), (r_1,y_1) \right) \leq L(c) \leq I_1 + I_2, 
$$
where 
\begin{align*}
I_1&= \int_0^1 \frac{\sqrt{k w (r \circ c)}}{r \circ c} |r_1-r_0| dt 
= \left| \int_0^1 \frac{\sqrt{k w (r \circ c)}}{r \circ c} (r_1-r_0) dt \right|
= \left|T(r_1)-T(r_0) \right|, \\ 
I_2 &= \int_0^1 \sqrt{w (r \circ c)} \left| \frac{d \tilde y}{d t} (t) \right|_{g_Y} dt. 
\end{align*}
Set $C''= \max \{ \sqrt{w(r)} \mid r \in [R_1, R_2] \}$. 
Since $(r \circ c)(t) = (r_1-r_0) t + r_0 \in (R_1,R_2)$ for any $t \in [0,1]$, we see that 
$$
I_2 \leq C'' L_{g_Y} (\tilde y). 
$$
Since $\tilde y$ is arbitrary, we obtain (2). 
\end{proof}

For a subset $S \subset \R_{>0} \times Y$, 
denote by ${\rm diam}_{d_g} (S)$ the diameter of $S$ w.r.t. $d_g$. 
Assuming the behaviors of $w(r)$ around $r=0$ and $\infty$, 
we have the following estimates. 
These are very useful to control the $d_g$-Cauchy sequences 
$\{ (r_k, y_k)\} \subset \R_{>0} \times Y$ 
with $\lim_{k \rightarrow \infty} r_k =0$ or $\infty$.

\begin{lemma}\label{lem:diam}
For any $(r_0, y_0), (r_1, y_1) \in \R_{>0} \times Y$, we have the following. 

\begin{enumerate}
\renewcommand{\labelenumi}{(\arabic{enumi})}
\item 
If $T_0 \in \R$ and $\lim_{r \rightarrow 0} w(r) =0$, we have  
\begin{align}\label{eq:estimate for diam 0} 
d_g \left((r_0, y_0), (r_1, y_1) \right) \leq T(r_0) + T(r_1) - 2 T_0. 
\end{align}
In particular, we have for  $R>0$
$$
{\rm diam}_{d_g} \{ (r,y) \in \R_{>0} \times Y \mid r \leq R \} \leq 2 (T(R)-T_0). 
$$

\item 
If $T_\infty \in \R$ and $\lim_{r \rightarrow \infty} w(r) =0$, we have 
\begin{align}\label{eq:estimate for diam inf} 
d_g \left((r_0, y_0), (r_1, y_1) \right) \leq 2 T_\infty - T(r_0) - T(r_1). 
\end{align}
In particular, we have for  $R>0$ 
$$
{\rm diam}_{d_g} \{ (r,y) \in \R_{>0} \times Y \mid r \geq R \} \leq 2 (T_\infty - T(R)). 
$$ 
\end{enumerate}
\end{lemma}

\begin{remark}
If $T_0 \in \R$, we easily see $\liminf_{r \rightarrow 0} w(r) =0$. 
However, $T_0 \in \R$ does not imply $\lim_{r \rightarrow 0} w(r) =0$. 

Indeed, setting $q=e^x$ for $x \in \R$ 
and defining $u:\R \rightarrow \R_{>0}$ by $u(x) = \sqrt{k w (e^x)}$, 
the condition $T_0 \in \R$ is equivalent to $\int_{-\infty}^{-1} u(x) dx < \infty$. 
Suppose that 
$u(x) = \frac{1}{x^2} +S(x)$ for $x \in (-\infty, -1]$, where $S: \R \rightarrow \R$ is given by 
$$
S(x) = \sum_{n \in \mathbb{Z} - \{ 0 \}} S_n (x), \qquad 
S_n (x) = 
\left\{ \begin{array}{ll}
n^2 \left( x-n+\frac{1}{n^2} \right)   & \mbox{for } x \in \left[ n-\frac{1}{n^2}, n \right], \\
n^2 \left( -x+n+\frac{1}{n^2} \right) &\mbox{for }  x \in \left[ n, n + \frac{1}{n^2} \right], \\
0 & \mbox{otherwise.} \\
\end{array} \right.
$$
Then we see that $\int_{-\infty}^{-1} u(x) dx < \infty$ and $\lim_{x \rightarrow -\infty} u(x) \neq 0$. 
Though the function $S$ is not smooth, 
we may replace $S$ with a smooth function which approximates $S$. 
Similar statement also holds for $T_\infty$. 
\end{remark}

\begin{proof}
For any path $c$ connecting $(r_0, y_0)$ and $(r_1, y_1)$, 
we have $d_g \left((r_0, y_0), (r_1, y_1) \right) \leq L_g (c)$. 
We will take the following path to show (\ref{eq:estimate for diam 0}) and 
(\ref{eq:estimate for diam inf}). 

Fixing $s >0$, define $c_1,c_2,c_3:[0,1] \rightarrow \R_{>0} \times Y$ by 
\begin{align*}
c_1(t)&=\left( \left( (s-1)t + 1 \right)r_0, y_0 \right), \\
c_2(t)&=\left( s \left( (r_1 - r_0) t + r_0 \right), \tilde{y}(t) \right), \\
c_3(t)&=\left( \left( (s-1) (1-t) + 1 \right)r_1, y_1 \right), 
\end{align*}
where $\tilde y: [0,1] \rightarrow Y$ is a path such that 
$\tilde y (0)=y_0$ and $\tilde y (1)=y_1$. 
That is, $c_1$ is a path connecting $(r_0, y_0)$ and $(s r_0, y_0)$, 
$c_2$ is a path connecting $(s r_0, y_0)$ and $(s r_1, y_1)$, and 
$c_3$ is a path connecting $(s r_1, y_1)$ and $(r_1, y_1)$. 
Define $c:[0,1] \rightarrow \R_{>0} \times Y$ by the concatenation of these paths: 
$$c = c_1 * c_2 * c_3.$$ 
Then we compute 
\begin{align*}
L_g(c_1) 
=& \int_0^1 r_0 |s-1| \frac{\sqrt{k w (r \circ c_1)}}{r \circ c_1} dt\\
=& \left| \int_0^1  r_0 (s-1) \frac{\sqrt{k w (r \circ c_1)}}{r \circ c_1} dt \right| 
= \left| \int_0^1 \frac{d}{dt} T((r \circ c_1)) dt \right| 
= |T(s r_0) - T(r_0)|. 
\end{align*}
Similarly, we obtain $L(c_3) = |T(s r_1) - T(r_1)|.$
We also have 
$$
L_g (c_2) \leq I_3 + I_4,
$$
where 
\begin{align*}
I_3 &= 
\int_0^1 |s (r_1- r_0)| \frac{\sqrt{k w (r \circ c_2)}}{r \circ c_2} dt
= \left| \int_0^1 \frac{d}{dt} T((r \circ c_2)) dt \right| 
= |T(s r_1) - T(s r_0)|, \\
I_4 &= \int_0^1 
\sqrt{ w (r \circ c_2) \left| \frac{d \tilde y}{d t} (t) \right|_{g_Y}^2 } dt
= 
\int_0^1 
\sqrt{w (r \circ c_2)} \left| \frac{d \tilde y}{d t} (t) \right|_{g_Y} dt. 
\end{align*}
Since $(r \circ c_2) (t) \in [ s {\rm min} \{r_0, r_1 \}, s {\rm max} \{r_0, r_1 \}]$ 
for any $t \in [0, 1]$, 
setting 
$$
C'''=C'''(s, r_0,r_1,w) = \max \left \{ \sqrt{w(r)} \mid r \in 
[s {\rm min} \{r_0, r_1 \}, s {\rm max} \{r_0, r_1 \}] \right \}, $$
we see that  
$$
I_4 \leq C''' L_{g_Y} (\tilde y). 
$$
Summarizing these estimates, we obtain 
\begin{align} \label{eq:diam estimate 0}
\begin{split}
d_g \left((r_0, y_0), (r_1, y_1) \right) 
\leq & 
|T(s r_0) - T(r_0)| + |T(s r_1) - T(r_1)| \\
& + |T(s r_1) - T(s r_0)| + C''' L_{g_Y} (\tilde y). 
\end{split}
\end{align}

Now suppose that $T_0 \in \R$ and $\lim_{r \rightarrow 0} w(r) =0$. 
Then we have $\lim_{s \rightarrow 0} C''' = 0$. 
Letting $s \rightarrow 0$ in (\ref{eq:diam estimate 0}), we obtain 
$$
d_g \left((r_0, y_0), (r_1, y_1) \right) 
\leq 
|T_0 - T(r_0)| + |T_0 - T(r_1)| = T(r_0) + T(r_1) -2 T_0. 
$$

Next, suppose that $T_\infty \in \R$ and $\lim_{r \rightarrow \infty} w(r) =0$. 
Then we have $\lim_{s \rightarrow \infty} C''' = 0$. 
Letting $s \rightarrow \infty$ in (\ref{eq:diam estimate 0}), we obtain 
$$
d_g \left((r_0, y_0), (r_1, y_1) \right) 
\leq 
|T_\infty - T(r_0)| + |T_\infty - T(r_1)| = 2 T_\infty -T(r_0) - T(r_1). 
$$
\end{proof}

From these lemmas, we can determine the metric completion of $\R_{>0} \times Y$ w.r.t. $d_g$. 

\begin{theorem}\label{thm:comp wp}
The metric completion $\overline{\R_{>0} \times Y}$ of $\R_{>0} \times Y$ 
w.r.t. the metric $d_g$ induced from the Riemannian metric $g=g(w)$ given in (\ref{eq:def gw}) 
is homeomorphic to the following. 
\begin{enumerate}
\renewcommand{\labelenumi}{(\arabic{enumi})} 
\item 
If $T_0 = - \infty$ and $T_\infty = \infty$, 
$$
\R_{>0} \times \overline{Y} \qquad \mbox{with the product topology. }
$$ 
\item 
If $T_0 \in \R$, $T_\infty = \infty$ and $\lim_{r \rightarrow 0} w(r) =0$, 
$$
(\{ 0 \} \cup \R_{> 0}) \times \overline{Y} / \left( \{ 0 \} \times \overline{Y} \right)
= \left(\R_{>0} \times \overline{Y} \right) \cup \{ * \}
$$
with the topology $\Oo_0$ given below. 
\item 
If $T_0 = - \infty$, $T_\infty \in \R$ and $\lim_{r \rightarrow \infty} w(r) =0$,  
$$
(\R_{> 0} \cup \{ \infty \} ) \times \overline{Y} / 
\left( \{ \infty \} \times \overline{Y} \right)
= \left(\R_{>0} \times \overline{Y} \right) \cup \{ * \} 
$$
with the topology $\Oo_\infty$ given below. 
\item 
If $T_0 \in \R$, $T_\infty \in \R$, $\lim_{r \rightarrow 0} w(r) =0$ and 
$\lim_{r \rightarrow \infty} w(r) =0$,  
$$
\left(\{ 0\} \cup \R_{> 0} \cup \{ \infty \} \right) \times \overline{Y} / 
\left( \{0, \infty \} \times \overline{Y} \right)
= \left(\R_{>0} \times \overline{Y} \right) \cup \{ * \} \cup \{ * \} 
$$
with the topology $\Oo_{0, \infty}$ given below. 

\end{enumerate}
Here, $\overline{Y}$ is the metric completion of $Y$ w.r.t. the metric $d_{g_Y}$ induced from $g_Y$. 
Let $\pi_0: 
(\{ 0 \} \cup \R_{> 0}) \times \overline{Y} \rightarrow 
(\{ 0 \} \cup \R_{> 0}) \times \overline{Y} / \left( \{ 0 \} \times \overline{Y} \right) 
$
be the projection. 
Set $*_0 = \pi_0 (\{ 0 \} \times \overline{Y})$. 
The topology $\Oo_0$ is defined by the 
fundamental system of neighborhoods $\Uu (x)$ given below. 
If $x \neq *_0$, $\Uu (x)$ consists of $\epsilon$-balls centered at $x$ for $\epsilon >0$ 
w.r.t. the product metric.   
If $x = *_0$, we set 
$$
\Uu (*_0) = \{ \pi_0 ([0, \epsilon) \times \overline{Y}) \mid \epsilon > 0 \}. 
$$

Let $\pi_\infty: 
(\R_{> 0} \cup \{ \infty \} ) \times \overline{Y} \rightarrow 
(\R_{> 0} \cup \{ \infty \} ) \times \overline{Y} / \left( \{ \infty\} \times \overline{Y} \right) 
$
be the projection. 
Set $*_\infty = \pi_\infty (\{ \infty \} \times \overline{Y})$. 
The topology $\Oo_\infty$ is given by the 
fundamental system of neighborhoods $\Uu (x)$ given below. 
If $x \neq *_\infty$, $\Uu (x)$ consists of $\epsilon$-balls centered at $x$ for $\epsilon >0$ 
w.r.t. the product metric.   
If $x = *_\infty$, we set 
$
\Uu (*_\infty) = \{ \pi_\infty ((R, \infty] \times \overline{Y}) \mid R > 0 \}. 
$
The topology $\Oo_{0, \infty}$ is similarly defined 
by setting the fundamental systems of neighborhoods as above. 
\end{theorem}

\begin{remark}
Roughly speaking, 
the metric completion is the cylinder of $\overline{Y}$ in the case (1), 
the cone (with the apex) of $\overline{Y}$ in the cases (2) and (3), 
and the suspension of $\overline{Y}$ in the case (4). 

In general, the topologies $\Oo_0, \Oo_\infty$ and $\Oo_{0, \infty}$ are 
weaker than the quotient topologies. 
If $\overline{Y}$ is compact, they agree with the quotient topologies. 
In particular, 
in the case (4), the metric completion $\overline{\R_{>0} \times Y}$ 
is compact if $\overline{Y}$ is compact 
because there is a surjection from 
$\left(\{ 0\} \cup \R_{> 0} \cup \{ \infty \} \right) \times \overline{Y} \cong [0,1] \times \overline{Y}$. 
\end{remark}

\begin{proof}
Use the notation of Definition \ref{def:comp}. 
Consider the case (1). Define a map 
\begin{align} \label{eq:bij comp}
\Theta_1: \overline{\R_{>0} \times Y} \rightarrow \R_{>0} \times \overline{Y}, 
\qquad 
[ (r_k, y_k) ] \mapsto \left(\lim_{k \rightarrow \infty} r_k, [ y_k ] \right).
\end{align}
This map is well-defined. 
Indeed, by Corollary \ref{cor:conv r}, we have $\lim_{k \rightarrow \infty} r_k \in \R_{>0}$. 
Then we may assume that $\{ (r_k, y_k) \} \subset (R_1, R_2) \times Y$ for some $0<R_1<R_2$. 
Then Lemma \ref{lem:R12} (1) implies that 
$\{ y_k \}$ is a $d_{g_Y}$-Cauchy sequence. 
If 
$\lim_{k \rightarrow \infty} d_g ((r_k, y_k), (r'_k, y'_k)) =0$ for 
$d_g$-Cauchy sequences $\{(r_k, y_k) \}$ and $\{(r'_k, y'_k) \}$, 
Lemma \ref{lem:unif conti} and Lemma \ref{lem:R12} (1) 
imply that $\lim_{k \rightarrow \infty} r_k = \lim_{k \rightarrow \infty} r'_k$ 
and $\lim_{k \rightarrow \infty} d_g (y_k, y'_k) = 0$, 
and hence $\Theta_1$ is well-defined. 

We show that $\Theta_1$ is bijective. 
For any $\left(r_0, [ y_k ] \right) \in \R_{>0} \times \overline{Y}$, 
$\{ (r_0, y_k) \}$ is a $d_g$-Cauchy sequence by Lemma \ref{lem:R12} (2). 
Hence we see that $\Theta_1$ is surjective. 
Suppose that 
$\lim_{k \rightarrow \infty} r_k = \lim_{k \rightarrow \infty} r'_k$ 
and $\lim_{k \rightarrow \infty} d_g (y_k, y'_k) = 0$ for 
$d_g$-Cauchy sequences $\{(r_k, y_k) \}$ and $\{(r'_k, y'_k) \}$. 
Then Lemma \ref{lem:R12} (2) implies that 
$\lim_{k \rightarrow \infty} d_g ((r_k, y_k), (r'_k, y'_k)) =0$, 
and hence $\Theta_1$ is injective. 

We show that $\Theta_1$ is homeomorphic. 
Let $\{ [(r_{k j}, y_{k j})] \}_j$ be a sequence in $\overline{\R_{>0} \times Y}$ 
converging to $[(r_{k}, y_{k})]$. 
That is, 
$$\lim_{j \rightarrow \infty} \lim_{k \rightarrow \infty} 
d_g \left( (r_{k j}, y_{k j}), (r_k, y_k) \right) =0. 
$$
By Lemma \ref{lem:unif conti}, we have 
$\lim_{j \rightarrow \infty} \lim_{k \rightarrow \infty} |r_{k j} - r_k| =0$. 
Since $\lim_{k \rightarrow \infty} r_k >0$, we can apply Lemma \ref{lem:R12} (1) and 
it follows that 
$
\lim_{j \rightarrow \infty} \lim_{k \rightarrow \infty} d_{g_Y} (y_{k j}, y_k) = 0. 
$
Hence $\Theta_1$ is continuous. 

Let $\{ (r_{0 j}, [y_{k j}]) \}_j$ be a sequence in $\R_{>0} \times \overline{Y}$ 
converging to $(r_0, [y_{k}])$. 
Since $\Theta_1^{-1}(r_0, [y_k]) = [(r_0, y_k)]$, Lemma \ref{lem:R12} (2) implies that 
$$
\lim_{j \rightarrow \infty} d_g \left([(r_{0 j}, y_{k j})], [(r_0, y_k)] \right)
=
\lim_{j \rightarrow \infty} \lim_{k \rightarrow \infty} d_g ((r_{0 j}, y_{k j}), (r_0, y_k)) = 0. 
$$
Hence $\Theta_1^{-1}$ is continuous. 
\\

Next, we consider the case (2). 
Define a map 
$\Theta_2: \overline{\R_{>0} \times Y} \rightarrow 
(\{ 0 \} \cup \R_{> 0}) \times \overline{Y} / \left( \{ 0 \} \times \overline{Y} \right)$ by 
\begin{align*}
\Theta_2 ([ (r_k, y_k) ]) 
= 
\left\{ \begin{array}{ll}
\left(\lim_{k \rightarrow \infty} r_k, [ y_k ] \right) & \mbox{if} \quad \lim_{k \rightarrow \infty} r_k >0, \\
*_0 & \mbox{if} \quad \lim_{k \rightarrow \infty} r_k = 0. \\
\end{array} \right.
\end{align*}
This map is well-defined and bijective. 
Indeed, 
Corollary \ref{cor:conv r} implies that $\lim_{k \rightarrow \infty} r_k \in \{ 0\} \cup \R_{> 0}$. 
Every $d_g$-Cauchy sequence with $\lim_{k \rightarrow \infty} r_k > 0$ 
corresponds to an element of $\R_{>0} \times \overline{Y}$ 
as in the case (1). 
For $d_g$-Cauchy sequences $\{(r_k, y_k) \}$ and $\{(r'_k, y'_k) \}$ 
such that $\lim_{k \rightarrow \infty} r_k = \lim_{k \rightarrow \infty} r'_k = 0$,  
Lemma \ref{lem:diam} (1) implies that 
$\lim_{k \rightarrow \infty} d_g ((r_k, y_k), (r'_k, y'_k)) =0$. 
Hence $\Theta_2$ is well-defined and bijective. 
 
We show that $\Theta_2$ is homeomorphic. 
Denote by $*$ the unique equivalence class $[(r_k, y_k)] \in \overline{\R_{>0} \times Y}$ 
such that $\lim_{k \rightarrow \infty} r_k =0$. 
By (1), we see that 
$\Theta_2|_{\overline{\R_{>0} \times Y} - \{ * \} }: 
\overline{\R_{>0} \times Y} - \{ * \} \rightarrow \R_{>0} \times \overline{Y}$ is homeomorphic. 
To prove the continuity of $\Theta_2$ at $*$, 
we prove the following. 

\begin{lemma}
The fundamental system of neighborhoods at  $*$ w.r.t. the topology 
induced from $d_g$ is given by  
$$
\{ U_\epsilon \mid \epsilon > 0 \} \qquad 
\mbox{where} \qquad 
U_\epsilon = \left\{ [(r_k, y_k)] \in \overline{\R_{>0} \times Y} 
\ \middle| \ \lim_{k \rightarrow \infty} r_k  < \epsilon \right \}. 
$$
\end{lemma}
\begin{proof}
Since $(\overline{\R_{>0} \times Y}, d_g)$ is a metric space, 
the fundamental system of neighborhoods at $*$ consists of 
the $\delta$-balls $B_\delta$ centered at $*$ for $\delta > 0$. 
Hence we only have to show that for any $\delta >0$, there exists $\epsilon >0$ 
such that $U_\epsilon \subset B_\delta$. 

Since the function $T$ in (\ref{eq:def T}) is continuous 
at $0$ under the assumption of (2), 
for any $\delta >0$, there exists $\epsilon >0$ such that 
$r < \epsilon \Rightarrow T(r) - T_0 < \delta$. Then (\ref{eq:estimate for diam 0}) 
implies that for any $[(r_k, y_k)] \in U_\epsilon$, 
$$
d_g (*, [(r_k, y_k)]) \leq \lim_{k \rightarrow \infty} T(r_k) - T_0 < \delta, 
$$
which implies that $U_\epsilon \subset B_\delta$.
\end{proof}

Then since 
$\Theta_2 (U_\epsilon) = \pi_0 ([0, \epsilon) \times \overline{Y})$, 
we see that $\Theta_2$ is continuous at $*$ 
and $\Theta_2^{-1}$ is continuous at $*_0$. 
We can prove (3) and (4) similarly. 
\end{proof}

Finally, we give a description of $\overline{Y}$ in terms of $\overline{\R_{>0} \times Y}$. 
The following implies that 
we can recover $\overline{Y}$ from $\overline{\R_{>0} \times Y}$. 

\begin{proposition}\label{prop:comp 1}
Use the notation of Definition \ref{def:comp}. 
For any $R>0$, the map 
$$
I_R:\overline{Y} \rightarrow \left\{ [ (r_k, y_k) ] \in \overline{\R_{>0} \times Y} 
\ \middle| \ \lim_{k \rightarrow \infty} r_k = R \right \}, 
\qquad 
[ y_k ] \mapsto [ (R, y_k) ]
$$
is homeomorphic.
\end{proposition}

\begin{proof}
The proof is similar to that of Theorem \ref{thm:comp wp}. 
Let $\{ y_k \}$ is a $d_{g_Y}$-Cauchy sequence. 
Then $\{ (R, y_k) \}$ is a $d_g$-Cauchy sequence by Lemma \ref{lem:R12} (2). 
Hence $I_R$ is well-defined. 

Let $\{(r_k, y_k) \}$ be a $d_g$-Cauchy sequence with 
$\lim_{k \rightarrow \infty} r_k = R$. 
Then Lemma \ref{lem:R12} (1) implies that 
$\{ y_k \}$ is a $d_{g_Y}$-Cauchy sequence. 
By Lemma \ref{lem:R12} (2), we have 
$$
d_g ((r_k, y_k), (R, y_k)) \leq |T(r_k) - T(R)| \rightarrow 0 
\qquad \mbox{as } \quad k \rightarrow \infty. 
$$
Then $\{(r_k, y_k) \} \sim \{(R, y_k) \}$, and hence $I_R$ is surjective.  

Suppose that 
$\lim_{k \rightarrow \infty} d_g ((R, y_k), (R, y'_k)) = 0$ 
for 
$d_{g_Y}$-Cauchy sequences $\{ y_k \}$ and $\{ y'_k \}$. 
Then Lemma \ref{lem:R12} (1) implies that 
$\lim_{k \rightarrow \infty} d_{g_Y} (y_k, y'_k) =0$, 
and hence $I_R$ is injective. 

We show that $I_R$ is homeomorphic. 
Let $\{ [y_{k j}] \}_j$ be a sequence in $\overline{Y}$ 
converging to $[y_{k}]$. 
Then by Lemma \ref{lem:R12} (2), 
$\lim_{j \rightarrow \infty} d_{g} ([(R, y_{k j}], [(R, y_k)]) = 0$, 
and hence $I_R$ is continuous.
By the proof above, we have $I_R^{-1}([(r_k, y_k)]) = [y_k]$. 
Then by Lemma \ref{lem:R12} (1), we see that $I_R^{-1}$ is continuous.
\end{proof}

\begin{remark}
Thus if we know $\overline{\R_{>0} \times Y}$, 
we see $\overline{Y}$. 
In particular, by Theorem \ref{thm:comp wp}, 
if we know $\overline{(\R_{>0} \times Y, d_{g(w)})}$, the metric completion of $\R_{>0} \times Y$
w.r.t. $d_{g(w)}$, for one $w$, 
we can obtain $\overline{(\R_{>0} \times Y, d_{g(\tilde w)})}$ for $\tilde w$ 
satisfying one of four assumptions in Theorem \ref{thm:comp wp}. 
\end{remark}


\section{Conformal transformations of the pseudo-Riemannian metric of a homogeneous pair} \label{sec:hp}
\subsection{The splitting theorem}

In this section, 
we give the definition of a homogeneous pair 
for a pseudo-Riemannian metric $g$ and a positive function $f$
on a manifold $M$ admitting a free $\R_{>0}$-action in more detail. 
Then we study the geometric structures of 
the pseudo-Riemannian manifold $(M, (v \circ f) g)$, where 
$v: \R_{>0} \rightarrow \R_{>0}$ is a smooth function.

\begin{definition}[Definition \ref{def:homog pair intro}] \label{def:homog pair} 
Let $(M,g)$ be a pseudo-Riemannian manifold 
which admits a free $\R_{>0}$-action. 
Denote by $m:\R_{>0} \times M \rightarrow M$ the $\R_{>0}$-action 
and set $m_\lambda = m(\lambda, \cdot)$ for $\lambda \in \R_{>0}.$ 
Let $P \in \mathfrak{X}(M)$ be a vector field generated by the $\R_{>0}$-action. That is, 
$$
P_x = \left. \frac{d}{dt} m(e^t, x) \right|_{t=0} 
$$
for $x \in M$. 
Suppose that 
$f: M \rightarrow \R_{>0}$ is a smooth function and 
$\alpha \in \R - \{ 0 \}$. 

The pair $(g, f)$ is called a 
{\bf homogeneous pair} of degree $\alpha$ if 
\begin{align}
m_\lambda^* g &= \lambda^\alpha g , \label{eq:homog g}\\
m_\lambda^* f &= \lambda^\alpha f, \label{eq:homog f}\\
g(P, \cdot) &= d f \label{eq:thm assump}
\end{align}
for any $\lambda > 0$. 
\end{definition}

\begin{remark}
The degree of $g$ must be equal to that of $f$. 
That is, if 
$m_\lambda^* g = \lambda^\alpha g$ and 
$m_\lambda^* f = \lambda^\beta f$, 
the equation $g(P, \cdot) = d f$ implies that $\alpha = \beta$. 

Indeed, by $g(P, \cdot) = d f$, we have 
$$
m_\lambda^* (g(P, \cdot)) = m_\lambda^* d f = \lambda^\beta d f. 
$$

Since 
$P_{m_\lambda(x)} = \left. \frac{d}{dt} m_{e^t} m_\lambda (x) \right|_{t=0} = (m_\lambda)_* P_{x}$ 
for any $x \in M$, 
we compute 
$$
m_\lambda^* (g(P, \cdot)) 
= (g \circ m_\lambda) (P \circ m_\lambda, (m_\lambda)_* (\cdot))
= (m_\lambda^* g)(P, \cdot)
= \lambda^\alpha g(P, \cdot)
= \lambda^\alpha d f.
$$
Hence we obtain $\alpha = \beta$. 
\end{remark}

We first show that $(M, (v \circ f) g)$ admits the structure of a warped product. 
This is a generalization of  
the splitting theorem for Hessian manifolds that are cones given in 
\cite[Theorem 1]{Loftin} and \cite[Lemmas 2.1 and 2.4]{Totaro} 
(cf. Remark \ref{rem:TotaroLoftin}).

\begin{theorem} \label{thm:split}
Let $(M,g)$ be a pseudo-Riemannian manifold 
which admits a free $\R_{>0}$-action 
and let $f: M \rightarrow \R_{>0}$ be a smooth function. 
Suppose that $(g, f)$ is a homogeneous pair of degree $\alpha$. 
Then 
\begin{enumerate}
\renewcommand{\labelenumi}{(\arabic{enumi})}
\item
we have 
$(df)_x \neq 0$ for any $x \in M$. Thus for any $l >0$ 
$$
M_l = \{ x \in M \mid f (x) = l \} 
$$
is a submanifold of $M$. 
Denote by $g_l$ the pullback of $g$ to $M_l$.  
Then $g_l$ is a pseudo-Riemannian metric on $M_l$.

\item For a function $v: \R_{>0} \rightarrow \R_{>0}$, 
the map  
\begin{align} \label{eq:psi}
\psi :\R_{>0} \times M_l \rightarrow M, \qquad 
(r, y) \mapsto m \left( \left(\frac{r}{l} \right)^{\frac{1}{\alpha}}, y \right) 
\end{align}
gives an isometry between 
$\left(\R_{>0} \times M_l, v (r) \left( \frac{1}{\alpha r} dr^2 
+ \frac{r}{l} g_l \right) \right)$
and $(M,(v \circ f) g)$. 
\end{enumerate}
\end{theorem}

\begin{remark} \label{rem:indep l}
For $l_1,l_2 >0$, the diffeomorphism 
$$
\psi_{l_1,l_2}: M_{l_1} \rightarrow M_{l_2}, \qquad x \mapsto 
m \left( \left( \frac{l_2}{l_1} \right)^{1/\alpha}, x \right) 
$$
gives an isometry $(M_{l_1}, g_{l_1}/l_1) \cong (M_{l_2}, g_{l_2}/l_2)$ 
by (\ref{eq:homog g}).  
Hence the isometry (\ref{eq:psi}) 
is independent of $l>0.$ 
\end{remark}

\begin{remark}
We do not use the local coordinates to prove Theorem \ref{thm:split}. 
Thus the statement formally holds when $M$ is infinite dimensional. 

The subtle point in the infinite dimensional case 
is the notion of submanifolds. 
In (1), we use implicit function theorem 
to prove that $M_l$ is a submanifold of $M$ by $(df)_x \neq 0$. 
However, 
there is no implicit function theorem in the infinite dimensional case in general. 
(For example, if $M$ is a Banach manifold, there is an implicit function theorem.) 
For the details of the theory of infinite dimensional manifolds, 
see \cite{KM, Lang}. 
\end{remark}

\begin{remark} \label{rem:hp wp}
By Theorem \ref{thm:split}, we have an isometry between 
$(M, (v \circ f) g)$ and 
$\left(\R_{>0} \times M_l, v(r) \left( \frac{1}{\alpha r} dr^2 + \frac{r}{l} g_l \right) \right)$. 
Setting 
$$k = \frac{l}{\alpha}, \qquad
w(r) =\frac{r v(r)}{l} \qquad \mbox{and} \qquad (Y,g_Y) = (M_l, g_l),$$  
this pseudo-Riemannian metric is of the form $g(w)$ in (\ref{eq:def gw}). 
\end{remark}

\begin{proof}[Proof of Theorem \ref{thm:split}]
First, we prove (1). 
For any $x \in M$, we compute 
\begin{align}\label{eq:df P}
(d f)_x (P_x) 
= (d f)_x \left( \left. \frac{d}{d t} m(e^t, x) \right|_{t=0} \right)
\stackrel{(\ref{eq:homog f})}
= \left. \frac{d}{d t} e^{\alpha t} f(x) \right|_{t=0}
= \alpha f(x). 
\end{align}
Since $f$ is a positive function, we see that $(df)_x \neq 0$. 
For $y \in M_l$, we have the decomposition
\begin{align} \label{eq:decomp}
T_y M = T_y M_l \oplus \R P_y = \ker (df)_y \oplus \R P_y, \qquad  
A= \left(A- \frac{(df)_y (A)}{\alpha l} P_y \right) + \frac{(df)_y (A)}{\alpha l} P_y, 
\end{align}
which is orthogonal by (\ref{eq:thm assump}). 
Then it is immediate to see that $g_l$ is a pseudo-Riemannian metric on $M_l$. \\

Next, we prove (2). 
Since the inverse $\psi^{-1}:M \rightarrow \R_{>0} \times M_l$ is given by 
\begin{align}\label{eq:psi inv}
\psi^{-1}(x) = \left(f(x), m \left(\left( \frac{l}{f(x)} \right)^{\frac{1}{\alpha}}, x \right) \right), 
\end{align}
we see that $\psi$ is a diffeomorphism. 
We compute $\psi^* ((v \circ f) g)$ to show that $\psi$ is an isometry. 
For $(r, y) \in \R_{>0} \times M_l$, we have  
$$
\psi^* (v \circ f) (r,y) 
= v \left(f \left(m \left( \left(\frac{r}{l} \right)^{\frac{1}{\alpha}}, y \right) \right) \right)
\stackrel{(\ref{eq:homog f})}
= v(r). 
$$
Thus we only have to compute 
$(\psi^* g)_{(r,y)} (\p_r, \p_r), (\psi^* g)_{(r,y)} (\p_r, a)$, and 
$(\psi^* g)_{(r,y)} (a, a)$ for any $a \in \ker (df)_y$. 
Since 
\begin{align*}
(\psi_*)_{(r,y)} (\p_r) 
&= 
\left. \frac{d}{d s} m \left( \left( \frac{r+s}{l}\right)^{\frac{1}{\alpha}}, y \right) \right|_{s=0}\\
&=
\left. \frac{d}{d s} 
m_{\left( \frac{r}{l} \right)^{\frac{1}{\alpha}}} \circ 
m_{\left( 1+\frac{s}{r}\right)^{\frac{1}{\alpha}}} (y) 
\right|_{s=0}
=
\frac{1}{\alpha r} (m_{\left( \frac{r}{l} \right)^{\frac{1}{\alpha}}})_* P_y, 
\end{align*}
we have by (\ref{eq:homog g}), (\ref{eq:thm assump}) and (\ref{eq:df P})
$$
(\psi^* g)_{(r, y)} (\p_r, \p_r) 
= \frac{1}{(\alpha r)^2} \frac{r}{l}  g(P_y,P_y)
= \frac{1}{(\alpha r)^2} \cdot \frac{r}{l} \cdot \alpha l
= \frac{1}{\alpha r}. 
$$
Since 
$$
(\psi_*)_{(r,y)} a = (m_{\left( \frac{r}{l} \right)^{\frac{1}{\alpha}}})_* a, 
$$
we have $(\psi^* g)_{(r,y)} (\p_r, a)=0$ by (\ref{eq:homog g}) and (\ref{eq:thm assump}). 
By (\ref{eq:homog g}), we obtain 
$$
(\psi^* g)_{(r,y)} (a, a) = \frac{r}{l} \cdot g(a, a).
$$
Hence the proof is completed. 
\end{proof}

Note that there is the following isometry between 
$(M, (v \circ f) g)$ and $(M, (\tilde v \circ f) g)$ for some $\tilde v: \R_{>0} \rightarrow \R_{>0}$. 
Hence they have the same sectional curvature, geodesics and 
the metric completion. 

\begin{lemma} \label{lem:symm homog met}
In the setting of Theorem \ref{thm:split}, 
the pseudo-Riemannian manifolds 
$(M, (v \circ f) g)$ and $\left( M, \frac{1}{f^2} v (\frac{1}{f}) g \right)$ 
are isometric via 
$\psi_M: M \rightarrow M$ defined by $\psi_M (x) = m(f(x)^{-2/\alpha}, x)$. 
\end{lemma}

\begin{proof}
Recall Remark \ref{rem:hp wp}. 
By Remark \ref{rem:symm gw}, 
we have an isometry 
\begin{align*}
\left(\R_{>0} \times M_l, v(r) \left( \frac{1}{\alpha r} dr^2 + \frac{r}{l} g_l \right) \right)
\cong &
\left(\R_{>0} \times M_l, v \left(\frac{1}{r} \right) 
\left( \frac{r}{\alpha} \left( d  \left(\frac{1}{r} \right) \right)^2 + \frac{1}{l r} g_l \right) \right)\\
= & 
\left(\R_{>0} \times M_l, 
\frac{1}{r^2} v \left(\frac{1}{r} \right) 
\left( \frac{1}{\alpha r} d r^2 + \frac{r}{l} g_l \right) \right)
\end{align*}
via $j: (r,y) \mapsto (1/r,y).$
Since the map $\psi$ in (\ref{eq:psi}) gives an isometry 
between 
$\left(\R_{>0} \times M_l, 
\frac{1}{r^2} v \left(\frac{1}{r} \right) \left( \frac{1}{\alpha r} d r^2 + \frac{r}{l} g_l \right) \right)
$
and $\left( M, \frac{1}{f^2} v (\frac{1}{f}) g \right)$, 
the map $\psi_M = \psi \circ j \circ \psi^{-1}: M \rightarrow M$ 
gives an isometry between $(M, (v \circ f) g)$ and $\left( M, \frac{1}{f^2} v (\frac{1}{f}) g \right)$. 
\end{proof}

\begin{definition}
Given a homogeneous pair $(g,f)$ of degree $\alpha \in \R - \{ 0,1 \}$, 
define a new pseudo-Riemannian metric $\hat g$ by 
\begin{align}\label{eq:def hat g}
\hat g = \frac{df \otimes df}{f} + (1 - \alpha) g. 
\end{align}
\end{definition}
As we see below, $(\hat g, f)$ is also a homogeneous pair of degree $\alpha$. 
This pseudo-Riemannian metric appears in many examples. 
See Sections \ref{sec:Hess} and \ref{sec:app}. 
The signature of $\hat g$ is different from that of $g$, 
and hence we can produce a definite pseudo-Riemannian metric 
in the sense of Definition \ref{def:def} in some cases. 

\begin{lemma}
The tensor $\hat g$ is a pseudo-Riemannian metric. 
The pair $(\hat g, f)$ is also a homogeneous pair of degree $\alpha$. 
\end{lemma}

\begin{proof}
Recalling the decomposition (\ref{eq:decomp}), 
suppose that $\hat g (k P + a, \cdot) = 0$ for $k \in \R$ and $a \in \ker (d f)$. 
Then we have 
$$
0 
= \hat g (k P + a, P) 
= \frac{k (df (P))^2}{f} + (1-\alpha) g(k P + a, P) 
\stackrel{(\ref{eq:thm assump}), (\ref{eq:df P})}
= k \alpha^2 f + k (1-\alpha) \alpha f 
= k \alpha f, 
$$
and hence we have $k=0$. 
Then it follows that $\hat g (a, \cdot) = (1-\alpha) g(a, \cdot)=0$, 
which implies that $a=0$. 
Hence $\hat g$ is a pseudo-Riemannian metric. 

It is clear to see that $m_\lambda^* \hat g = \lambda^\alpha \hat g$. 
By (\ref{eq:thm assump}) for $(g, f)$ and (\ref{eq:df P}), 
we see that $\hat g(P, \cdot) = df$. 
\end{proof}

By the definition of $\hat g$, we see that 
$\hat g_l = (1-\alpha) g_l$. Then by Theorem \ref{thm:split}, we have an isometry 
$$
(M, \hat g) \cong 
\left(\R_{>0} \times M_l, \frac{1}{\alpha r} dr^2 + (1-\alpha) \frac{r}{l} g_l \right). 
$$
Comparing this decomposition with 
$(M, g) \cong \left(\R_{>0} \times M_l, \frac{1}{\alpha r} dr^2 + \frac{r}{l} g_l \right)$, 
the definiteness of $\hat g$ is characterized in terms of the signature of $g$ as follows. 

\begin{lemma}
Setting $n= \dim M$, we have the following. 
\begin{enumerate}
\renewcommand{\labelenumi}{(\arabic{enumi})}
\item 
When $\alpha >1$, $g$ has signature $(1,n-1)$ if and only if 
$\hat g$ is positive definite.

\item
When $0 < \alpha <1$, $g$ is positive definite if and only if 
$\hat g$ is positive definite. 

\item
When $\alpha <0$, $g$ is negative if and only if 
$\hat g$ is negative definite. 
\end{enumerate}
\end{lemma}

\subsection{The sectional curvature} \label{subsec:csc met}

Let $(g, f)$ be a homogeneous pair on a manifold $M$. 
By Remark \ref{rem:hp wp}, we can apply results in Section \ref{sec:wp}. 
First, by Proposition \ref{prop:tech} (3), 
we can find a function $v: \R_{>0} \rightarrow \R_{>0}$ 
such that $(v \circ f) g$ has the constant sectional curvature 
if the level set $(M_l, g_l)$ has constant sectional curvature.

\begin{proposition} \label{prop:conf csc}
Use the notation of Definition \ref{def:w} and Theorem \ref{thm:split}. 
Let $(g, f)$ be a homogeneous pair of degree $\alpha$. 
Suppose that $g_l$  
has constant sectional curvature $\hat C_l \in \R$: $K^{g_l} = \hat C_l$.

Then given $C \in \R$, 
$(v \circ f) g$ has constant sectional curvature $C$ if 
$$
v(r) = \frac{1}{r} w \left(\frac{C}{\alpha}, \frac{l \hat C_l}{\alpha}, C_2, r \right), 
$$ 
where $C_2 \in \R$ such that 
$
\left(\frac{C}{\alpha}, \frac{l \hat C_l}{\alpha}, C_2 \right) 
\in \Delta_1 \cup \Delta_2 \cup \Delta_3$. 
\end{proposition}

\begin{proof}
By Remark \ref{rem:hp wp} and Proposition \ref{prop:tech} (3), 
$(v \circ f) g$ has constant sectional curvature $C$ if 
$$
v(r) 
= \frac{l}{r} w \left(\frac{l C}{\alpha}, C_1, C_2, r \right) 
= \frac{1}{r} w \left(\frac{C}{\alpha}, C_1, C_2, r \right) 
\qquad \mbox{and} \qquad 
\frac{C_1 \alpha}{l} = \hat C_l, 
$$ 
where 
$C_2 \in \R$ such that 
$
\left(\frac{l C}{\alpha}, C_1, C_2 \right) 
\in \Delta_1 \cup \Delta_2 \cup \Delta_3$. 
The last equation of $v(r)$ follows by Definition \ref{def:w}. 
\end{proof}

\begin{remark}
Remark \ref{rem:indep l} implies that $l K^{g_l}$ is independent of $l>0$ 
because $K^{g_l/l} = l K^{g_l}.$
Thus if $K^{g_l} = \hat C_l$, $l \hat C_l$ is independent of $l>0$. 

The function $v(r)$ given in Proposition \ref{prop:conf csc} 
is defined for all $r>0$ when $\alpha C \geq 0$. 
When $\alpha C < 0$, it is only defined 
on the complement of the discrete set of $\R_{>0}$. 
\end{remark}

To apply Proposition \ref{prop:conf csc}, 
$(M_l, g_l)$ needs to have constant sectional curvature. 
This is the case if $g$ is flat. 
The following is a generalization of \cite[Corollaries 2.2 and 2.3]{Totaro}. 

\begin{lemma} \label{lem:sc flat}
Use the notation of Theorem \ref{thm:split}. 
Let $(g, f)$ be a homogeneous pair of degree $\alpha$. 

\begin{enumerate}
\renewcommand{\labelenumi}{(\arabic{enumi})}
\item
We have 
$$
K^g (a,b) = \frac{l}{r} \left( K^{g_l} (a,b) - \frac{\alpha}{4 l} \right)
$$
for linearly independent $a,b \in TM_l$.
\item 
The pseudo-Riemannian metric $g$ is flat if and only if 
$g_l$ has constant sectional curvature $\frac{\alpha}{4 l}.$ 
\end{enumerate}
\end{lemma}

\begin{proof}
Suppose that $v=1$ in Remark \ref{rem:hp wp}. 
Since $w(r) = r/l = w(0,1/4,-\log l,r)$ in the notation of Definition \ref{def:w}, 
the statement follows from Proposition \ref{prop:tech} (2) and (3). 
\end{proof}

The following is immediate from 
Definition \ref{def:w}, Proposition \ref{prop:conf csc} and 
Lemma \ref{lem:sc flat}. 
The flatness of $g/f^2$ is also implied by Lemma \ref{lem:symm homog met}.

\begin{corollary} \label{cor:conf csc flat}
Let $(g, f)$ be a homogeneous pair of degree $\alpha$. 
Suppose that $g$ is flat. 
Then the following holds.  
\begin{itemize}
\item 
For $C \in \R$ such that $\alpha C >0$, set 
$$
v(r)= \frac{\alpha}{4 C r \left(\cosh (\frac{1}{2} (\log r + C_2)) \right)^2} 
\qquad (C_2 \in \R). 
$$
Then $(v \circ f) g$ has constant sectional curvature $C$. 

\item 
The pseudo-Riemannian metric $g/f^2$ is flat  on $M$. 
\end{itemize}
\end{corollary}

If $g$ is definite in the sense of Definition \ref{def:def} 
and the bound of the sectional curvature of $g_l$ is given, 
we can give the bounds of the sectional curvature of $g$.

\begin{corollary} \label{cor:bound Hess}
Use the notation of Definitions \ref{def:w} and \ref{def:def}. 
Let $(g, f)$ be a homogeneous pair of degree $\alpha$. 
Suppose that $g$ is definite. 
Given $C \in \R$, set 
$$
v(r) 
= \frac{1}{r} w \left(\frac{C}{\alpha}, C_1, C_2, r \right), 
$$ 
where 
$C_1, C_2 \in \R$ such that 
$
\left(\frac{C}{\alpha}, C_1, C_2 \right) 
\in \Delta_1 \cup \Delta_2 \cup \Delta_3$. 
Then 
\begin{itemize}
\item
$K^{(v \circ f) g} \geq C$ when $l K^{g_l} \geq \alpha C_1$, and 
\item
 $K^{(v \circ f) g} \leq C$ when $l K^{g_l} \leq \alpha C_1$. 
\end{itemize}
Furthermore, 
$K^{(v \circ f) g} = C$ 
if $l K^{g_l} = \alpha C_1$. 
\end{corollary}

\begin{proof}
Suppose that $v(r) = \frac{1}{r} w \left(\frac{C}{\alpha}, C_1, C_2, r \right)$ 
in Remark \ref{rem:hp wp}. 
Then we have 
$$w(r) =\frac{r v(r)}{l} = \frac{1}{l} w \left(\frac{C}{\alpha}, C_1, C_2, r \right)
= w \left(\frac{l C}{\alpha}, C_1, C_2, r \right), $$  
and hence the statement follows by Corollary \ref{cor:sc bound wp}. 
\end{proof}

\begin{remark}
In particular, we can apply this when $g$ is flat. 
By Lemma \ref{lem:sc flat}, this is the case $K^{g_l} = \frac{\alpha}{4 l}$. 

More generally, if $K^g \geq 0$ (resp. $\leq 0$), we have 
$l K^{g_l} \geq \frac{\alpha}{4}$ (resp. $\leq \frac{\alpha}{4}$). 
Then by Corollary \ref{cor:bound Hess}, 
we have $K^{(v \circ f) g} \geq C$ (resp. $\leq C$)
for 
$v(r) 
= \frac{1}{r} w \left(\frac{C}{\alpha}, \frac{1}{4}, C_2, r \right)$, 
where $C_2 \in \R$ such that 
$
\left(\frac{C}{\alpha}, \frac{1}{4}, C_2 \right) \in \Delta_1 \cup \Delta_2 \cup \Delta_3$. 
\end{remark}

When $\dim M =2$, we do not need the 
assumption on $K^{g_l}$ by Corollary \ref{cor:2dim csc wp}. 
We can prove the following in the same way as Proposition \ref{prop:conf csc}. 

\begin{corollary} \label{cor:2dim csc Hess}
Use the notation of Definition \ref{def:w}. 
Let $(g, f)$ be a homogeneous pair of degree $\alpha$. 
Suppose that $\dim M =2$.

Then given $C \in \R$, 
$(v \circ f) g$ has constant sectional curvature $C$ if 
$$
v(r) = \frac{1}{r} w \left(\frac{C}{\alpha}, C_1, C_2, r \right), 
$$ 
where $C_1, C_2 \in \R$ such that 
$
\left(\frac{C}{\alpha}, C_1, C_2 \right) 
\in \Delta_1 \cup \Delta_2 \cup \Delta_3$. 
In particular, setting $C=0$, we see that 
$f^\beta g$ is flat for any $\beta \in \R$. 
\end{corollary}

Corollary \ref{cor:2dim csc Hess} implies the following, 
which is a generalization of \cite[Section 6]{Totaro} for Hessian manifolds.

\begin{remark} \label{rem:2dim decomp}
Suppose that 
$M=M_1 \times \cdots \times M_k$, 
where $\dim M_i \leq 2$ for any $i$. 
If $(g_i, f_i)$ is a homogeneous pair of degree $\alpha$ on $M_i$, 
$(g, f) = (g_1 + \cdots + g_k, f_1 + \cdots + f_k)$
is a homogeneous pair on $M$. 
Then 
we can construct constant sectional curvature pseudo-Riemannian metrics on $M$ by Corollary \ref{cor:2dim csc Hess}.
In particular, $g$ is flat. \\
\end{remark}

Now recall the pseudo-Riemannian metric $\hat g$ defined in (\ref{eq:def hat g}). 
Since $\hat g_l = (1- \alpha) g_l$, 
$\hat g_l$ has constant sectional curvature if $g_l$ does. 
In particular, 
we can further obtain the following in addition to Corollary \ref{cor:conf csc flat}.

\begin{corollary} \label{cor:conf csc Hessian}
Let $(g, f)$ be a homogeneous pair of degree $\alpha \in \R - \{ 0, 1 \}$. 
Suppose that $g$ is flat. Then the following holds.  

\begin{itemize}
\item When $\alpha <1$, 
\begin{itemize}
\item
for $C \in \R$ such that $\alpha C > 0$, set 
$$
v(r)=
\frac{\alpha}{4 (1 - \alpha) C \left(\cosh (\frac{1}{2 \sqrt{1 - \alpha}} (\log r + C_2)) \right)^2} 
\qquad (C_2 \in \R). 
$$
Then $(v \circ f) \hat g$ is a pseudo-Riemannian metric on $M$ 
which has constant sectional curvature $C$. 
\item
The pseudo-Riemannian metric $f^{\pm \frac{1}{\sqrt{1 - \alpha}}} \hat g$ is flat. 
\end{itemize}
\item When $\alpha >1$, 
set for $C < 0$ 
$$
v(r)=
\frac{\alpha}{4 (1- \alpha) C \left( \sin (\frac{1}{2 \sqrt{\alpha -1}} (\log r + C_2)) \right)^2} 
\qquad (C_2 \in \R).
$$
Then $(v \circ f) \hat g$ is a pseudo-Riemannian metric 
which has constant sectional curvature $C$ 
defined on 
$M - \bigcup_{N \in \mathbb{Z}} 
f^{-1} \left(\exp \left(2N \sqrt{\alpha - 1} \pi - C_2 \right) \right)$. 
\end{itemize}
\end{corollary}

\begin{proof} 
Since $g$ is flat, we have $K^{g_l} = \frac{\alpha}{4 l}$ 
by Lemma \ref{lem:sc flat} (2). 
Since $\hat g_l = (1- \alpha) g_l$,  
it follows that  
$K^{\hat g_l} = \frac{\alpha}{4 (1 - \alpha) l}$, 
and hence 
$\frac{l K^{\hat g_l}}{\alpha} = \frac{1}{4 (1- \alpha)}$.  
Then by Proposition \ref{prop:conf csc}, 
it is straightforward to obtain the statement. 
\end{proof}

Finally, we give an application of Corollary \ref{cor:bound Hess}. 

\begin{corollary} \label{cor:bound log flat}
Let $(g, f)$ be a homogeneous pair of degree $\alpha >1$.  
Suppose further that
$K^{g} \geq 0$, and $\hat g$ is definite in the sense of Definition \ref{def:def}. 
Then we have 
$$
K^{f^\beta \hat g} \leq 0 \qquad \mbox{ for any } \beta \in \R. 
$$
\end{corollary}

\begin{proof}
Since 
$K^{g} \geq 0$, 
Lemma \ref{lem:sc flat} (1) implies that 
$K^{g_l} \geq \frac{\alpha}{4 l}$.  
Since $\hat g_l = (1- \alpha) g_l$,  
it follows that  
$l K^{\hat g_l} = \frac{l K^{g_l}}{1- \alpha} \leq \frac{\alpha}{4 (1- \alpha)} <0$. 
 
On the other hand, for any $\beta \in \R$, 
we have $r^\beta = w(0, |\beta|/4, 0, r)$, where we use the notation of Definition \ref{def:w}. 
Since $\alpha \cdot |\beta|/4 \geq 0$, Corollary \ref{cor:bound Hess} implies that 
$K^{f^\beta \hat g} \leq 0$.
\end{proof}

\subsection{The geodesics}

If $v(r)=r^\beta$, where $\beta \in \R$, 
we can describe the geodesics of $(M, f^\beta g)$ in terms of those in $(M_l, g_l)$ 
by Proposition \ref{prop:geod wp explicit}.

\begin{proposition} \label{prop:geod explicit}
Let $(g, f)$ be a homogeneous pair of degree $\alpha$. 
The geodesic $\gamma: (-\epsilon, \epsilon) \rightarrow  M$ 
with the initial position $x_0 \in M_l \subset M$ and 
the initial velocity $A \in T_{x_0} M$ 
w.r.t. the pseudo-Riemannian metric $f^\beta g$, where $\beta \in \R$, 
is given as follows. 

\begin{itemize}
\item
When $\beta \neq -1$, 
$$
\gamma (t) = m \left( \mu(\beta,t)^{\frac{1}{\alpha (\beta +1)}}, \ 
\hat y_l \left( \int_0^t \frac{d \tau}{\mu (\beta, \tau)} \right) \right), 
$$
where 
$$
\mu (\beta,t)= 1+ \frac{df (A)}{l} (\beta +1) t + \frac{\alpha}{4 l} g(A,A) (\beta +1)^2 t^2
$$
and  $\hat y_l (s)$ is the geodesic in $(M_l, g_l)$ 
with the initial position $x_0 \in M_l$ and the initial velocity $A-\frac{df(A)}{\alpha l} P \in T_{x_0} M_l$, 
the $T_{x_0} M_l$ component of $A$ in the decomposition (\ref{eq:decomp}). 
\item
When $\beta=-1$, 
$$
\gamma (t) = m \left( e^{\frac{df (A)}{\alpha l} t}, \hat y_l (t) \right). 
$$
\end{itemize}
\end{proposition}

Note that the integral $\int_0^t \frac{d \tau}{\mu (\beta, \tau)}$ 
can be explicitly computed as in Remark \ref{rem:explicit geod}.

\begin{proof}
By Theorem \ref{thm:split}, 
the geodesic $\gamma (t)$ is given by 
$$\gamma(t)=\psi (r(t), y(t))=m \left( \left(\frac{r(t)}{l} \right)^{1/\alpha}, y(t) \right),$$ 
where $(r(t), y(t))$ is a geodesic 
of $\left(\R_{>0} \times M_l, \frac{r^{\beta -1}}{\alpha} dr^2 + \frac{r^{\beta +1}}{l} g_l \right)$
with the initial position $\psi^{-1} (x_0)$ 
and the initial velocity $(d \psi^{-1})_{x_0} (A)$. 
By (\ref{eq:psi inv}), we see that 
$$
\psi^{-1} (x_0) = (l, x_0), \qquad
(d \psi^{-1})_{x_0} (A) = \left(df(A), A-\frac{df(A)}{\alpha l} P \right). 
$$
Since 
the Levi-Civita connection is invariant under the scalar multiplication of a pseudo-Riemannian metric, 
$(r(t), y(t))$ is a geodesic of 
$\frac{l r^{\beta-1}}{\alpha} dr^2 + r^{\beta +1} g_l$, 
which is of the form (\ref{eq:geod met}) 
if we set $k = \frac{l}{\alpha}, C_0= \beta+1$ and $(Y,g_Y) = (M_l, g_l)$.
Then the geodesic $(r(t), y(t))$ is given by Proposition \ref{prop:geod wp explicit}. 
Since 
\begin{align*}
g \left(A-\frac{df(A)}{\alpha l} P, A-\frac{df(A)}{\alpha l} P \right)
=&
g(A,A) -2 \frac{df(A)}{\alpha l} g(A,P) + \left(\frac{df(A)}{\alpha l} \right)^2 g(P,P)\\
\stackrel{(\ref{eq:thm assump}), (\ref{eq:df P})}
=&
g(A,A)-2 \frac{df(A)}{\alpha l} \cdot df(A) 
+ \left(\frac{df(A)}{\alpha l} \right)^2 \cdot \alpha l \\
=&
g(A,A) - \frac{1}{\alpha l} (df(A))^2,
\end{align*}
$F$ in Proposition \ref{prop:geod wp explicit} is given by 
\begin{align*}
F = \frac{1}{4} \left( \left( \frac{df(A)}{l} \right)^2 
+ \frac{\alpha}{l} \left( g(A,A) - \frac{1}{\alpha l} (df(A))^2 \right) \right)
=
\frac{\alpha}{4 l} g(A,A). 
\end{align*}
Hence the proof is completed. 
\end{proof}

Corollary \ref{cor:convex wp} implies the geodesically convexity or concavity of $f$ 
in the following cases. 

\begin{proposition}\label{prop:convex}
Let $(g, f)$ be a homogeneous pair of degree $\alpha$. 

\begin{enumerate}
\renewcommand{\labelenumi}{(\arabic{enumi})}
\item 
The function $f$ is geodesically convex w.r.t. $f^\beta g$ 
if one of the following condition holds. 
\begin{itemize}
\item
$\beta=-1$.
\item
$-1 < \beta \leq 1$ and $\alpha g_l$ is positive definite.
\item
$\beta < -1$ and  $\alpha g_l$ is negative definite. 
\end{itemize}
\item
The function $f$ is geodesically concave w.r.t. $f^\beta g$ 
if 
$\beta \geq 1$ and $\alpha g_l$ is negative definite. 
\end{enumerate}
\end{proposition}

\begin{proof}
By Theorem \ref{thm:split}, 
any geodesic $\gamma$ w.r.t. $f^\beta g$ is of the form 
$$
\gamma(t)=\psi (r(t), y(t))=m \left( \left(\frac{r(t)}{l} \right)^{1/\alpha}, y(t) \right),
$$ 
where $(r(t), y(t))$ is a geodesic 
of $\left(\R_{>0} \times M_l, \frac{l r^{\beta-1}}{\alpha} dr^2 + r^{\beta +1} g_l \right)$. 
Then we see that 
$$
f (\gamma (t)) = \frac{r(t)}{l} f(y(t)) = r(t), 
$$
where we use (\ref{eq:homog f}) and the fact that $y(t) \in M_l$. 
Then (1) and (2) hold from Corollary \ref{cor:convex wp}. 
\end{proof}

\subsection{The metric completion}
Let $(g, f)$ be a homogeneous pair. 
Use the notation of Theorem \ref{thm:split}. 
In this subsection, we assume the following. 
\begin{itemize}
\item
The pseudo-Riemannian metric $g$ is positive definite. 
\item
The pseudometric $d_g$ induced from $g$
is a metric.
(This is always true when $M$ is finite dimensional. 
In the infinite dimensional case, 
there are examples of a Riemannian metric 
whose induced pseudometric is identically zero (\cite{MM}). )
\end{itemize}
Then we study the metric completion of $M$ w.r.t. $d_{(v \circ f) g}$, 
where $d_{(v \circ f) g}$ is the metric induced from 
a Riemannian metric $(v \circ f) g$ for a function $v: \R_{>0} \rightarrow \R_{>0}$.

Fixing $R_0 \in \R_{>0}$, define a strictly increasing function $\hat T:\R_{>0} \rightarrow \R$ by 
\begin{align*}
\hat T (r)=\int^r_{R_0} \sqrt{\frac{v(q)}{q}} dq 
\end{align*}
and set 
\begin{align} \label{eq:tilde T0 Tinf}
\hat T_0 = \lim_{r \rightarrow 0} \hat T(r) \in \R \cup \{ - \infty \}, \qquad
\hat T_\infty = \lim_{r \rightarrow \infty} \hat T(r) \in \R \cup \{ \infty \}. 
\end{align}
Then we obtain the following by Remark \ref{rem:hp wp} and Theorem \ref{thm:comp wp}. 

\begin{theorem}\label{thm:comp}
The metric completion $\overline{M}$ of $M$ 
w.r.t. the metric $d_{(v \circ f) g}$ induced from the Riemannian metric $(v \circ f) g$ 
is homeomorphic to the following. 
\begin{enumerate}
\renewcommand{\labelenumi}{(\arabic{enumi})} 
\item 
If $\hat T_0 = - \infty$ and $\hat T_\infty = \infty$, 
$$
\R_{>0} \times \overline{M_l} \qquad \mbox{with the product topology. }
$$ 
\item 
If $\hat T_0 \in \R$, $\hat T_\infty = \infty$ and $\lim_{r \rightarrow 0} r v(r) =0$, 
$$
(\{ 0 \} \cup \R_{> 0}) \times \overline{M_l} / \left( \{ 0 \} \times \overline{M_l} \right)
= \left(\R_{>0} \times \overline{M_l} \right) \cup \{ * \} 
$$
with the topology $\Oo_0$ given below. 
\item 
If $\hat T_0 = - \infty$, $\hat T_\infty \in \R$ and $\lim_{r \rightarrow \infty} r v(r) =0$,  
$$
(\R_{> 0} \cup \{ \infty \} ) \times \overline{M_l} / 
\left( \{ \infty \} \times \overline{M_l} \right)
= \left(\R_{>0} \times \overline{M_l} \right) \cup \{ * \} 
$$
with the topology $\Oo_\infty$ given below. 
\item 
If $\hat T_0 \in \R$, $\hat T_\infty \in \R$, $\lim_{r \rightarrow 0} r v(r) =0$ and 
$\lim_{r \rightarrow \infty} r v(r) =0$,  
$$
\left(\{ 0\} \cup \R_{> 0} \cup \{ \infty \} \right) \times \overline{M_l} / 
\left( \{0, \infty \} \times \overline{M_l} \right)
= \left(\R_{>0} \times \overline{M_l} \right) \cup \{ * \} \cup \{ * \}
$$
with the topology $\Oo_{0, \infty}$ given below. 
\end{enumerate}
Here, $\overline{M_l}$ is the metric completion of $M_l$ w.r.t. 
the metric induced from $g_l$. 

Let $\pi_0: 
(\{ 0 \} \cup \R_{> 0}) \times \overline{M_l} \rightarrow 
(\{ 0 \} \cup \R_{> 0}) \times \overline{M_l} / \left( \{ 0 \} \times \overline{M_l} \right) 
$
be the projection. 
Set $*_0 = \pi_0 (\{ 0 \} \times \overline{M_l})$. 
The topology $\Oo_0$ is defined by the 
fundamental system of neighborhoods $\Uu (x)$ given below. 
If $x \neq *_0$, $\Uu (x)$ consists of $\epsilon$-balls centered at $x$ for $\epsilon >0$ 
w.r.t. the product metric.   
If $x = *_0$, 
$$
\Uu (*_0) = \{ \pi_0 ([0, \epsilon) \times \overline{M_l}) \mid \epsilon > 0 \}. 
$$

Let $\pi_\infty: 
(\R_{> 0} \cup \{ \infty \} ) \times \overline{M_l} \rightarrow 
(\R_{> 0} \cup \{ \infty \} ) \times \overline{M_l} / \left( \{ \infty\} \times \overline{M_l} \right) 
$
be the projection. 
Set $*_\infty = \pi_\infty (\{ \infty \} \times \overline{M_l})$. 
The topology $\Oo_\infty$ is defined by the 
fundamental system of neighborhoods $\Uu (x)$ given below. 
If $x \neq *_\infty$, $\Uu (x)$ consists of $\epsilon$-balls centered at $x$ for $\epsilon >0$ 
w.r.t. the product metric.   
If $x = *_\infty$, we set 
$
\Uu (*_\infty) = \{ \pi_\infty ((R, \infty] \times \overline{M_l}) \mid R > 0 \}. 
$
The topology $\Oo_{0, \infty}$ is similarly defined 
by setting the fundamental systems of neighborhoods as above. 
\end{theorem}

\begin{remark} \label{rem:comp rough}
By Remark \ref{rem:indep l}, $\overline{M_{l_1}}$ and $\overline{M_{l_2}}$ 
are isometric for $l_1, l_2 >0$. 
Thus Theorem \ref{thm:comp} is independent of $l$. 

Roughly speaking, 
the metric completion is the cylinder of $\overline{M_l}$ in the case (1), 
the cone (with the apex) of $\overline{M_l}$ in the cases (2) and (3), 
and the suspension of $\overline{M_l}$ in the case (4). 
In general, the topologies $\Oo_0, \Oo_\infty$ and $\Oo_{0, \infty}$ are 
weaker than the quotient topologies. 
If $\overline{M_l}$ is compact, they agree with the quotient topologies. 
In particular, 
in the case (4), the metric completion $\overline{M}$ 
is compact if $\overline{M_l}$ is compact. 
\end{remark}

By Proposition \ref{prop:comp 1}, we also obtain the following.

\begin{proposition}\label{prop:comp level}
Use the notation of Definition \ref{def:comp}. The map 
$$
\overline{M_l} \rightarrow 
\left\{ [ x_k ] \in \overline{M} 
\ \middle| \ \lim_{k \rightarrow \infty} f(x_k) = l \right \}, \qquad 
[ y_k ] \mapsto [ y_k ]
$$
is homeomorphic. 
\end{proposition}

\begin{proof}
By Proposition \ref{prop:comp 1}, the map 
$$
\overline{M_l} \rightarrow \left\{ [ (r_k, y_k) ] \in \overline{\R_{>0} \times M_l} 
\ \middle| \ \lim_{k \rightarrow \infty} r_k = l \right \}, 
\qquad 
[ y_k ] \mapsto [ (l, y_k) ]
$$
is homeomorphic. Since $\psi$ in (\ref{eq:psi}) is isometric, the map 
$$
\overline{\R_{>0} \times M_l} \rightarrow \overline{M}, \qquad 
[ (r_k ,y_k) ] \mapsto [ \psi (r_k, y_k) ]
$$
is isometric. Since $r_k = f(\psi (r_k, y_k))$ and $\psi (l, y_k) = y_k$, 
the proof is completed. 
\end{proof}

\begin{remark}
Thus if we know $\overline{M}$, we see $\overline{M_l}$. 
In particular, by Theorem \ref{thm:comp}, 
if we know $\overline{(M, d_{(v \circ f) g})}$, the metric completion of $M$
w.r.t. $d_{(v \circ f) g}$,  for one $v$, 
we obtain $\overline{(M, d_{(\tilde v \circ f) g})}$ for $\tilde v$ 
satisfying one of four assumptions in Theorem \ref{thm:comp}. 
\end{remark}

\section{Pseudo-Hessian manifolds} \label{sec:Hess}

Theorem \ref{thm:split} applies to many important classes of pseudo-Riemannian manifolds. 
One of them is the following class, which includes 
a class of pseudo-Hessian manifolds satisfying the conditions (\ref{eq:homog})--(\ref{eq:def P}).

\begin{proposition} \label{prop:split met}
Let $M$ be a manifold 
admitting a torsion-free connection $D$, 
a function $f: M \rightarrow \R_{>0}$ such that 
$h=Ddf$ is a pseudo-Riemannian metric 
(If $D$ is flat, $h$ is called a pseudo-Hessian metric), 
and  
a free $\R_{>0}$-action $m:\R_{>0} \times M \rightarrow M$. 
Set $m_\lambda =m(\lambda, \cdot)$ for $\lambda >0$. 
Suppose the following.

\begin{itemize}
\item 
The function $f: M \rightarrow \R_{>0}$ is homogeneous of degree $\alpha \in \R$:  
\begin{align} \label{eq:homog}
m_\lambda^* f = \lambda^\alpha f \qquad \mbox{for any } \lambda >0.
\end{align}
\item The action of $\R_{>0}$ preserves $D$: 
That is,  
\begin{align} \label{eq:invariant}
D_{(m_\lambda)_* A} \left( (m_\lambda)_* B \right) = (m_\lambda)_* (D_A B) 
\end{align}
for any $\lambda > 0$ and vector fields $A, B \in \mathfrak{X}(M)$ 
(cf. \cite[Chapter VI, Proposition 1.4]{KN}). 
\item 
For a vector field $P \in \mathfrak{X}(M)$ generated by the $\R_{>0}$-action, 
we have 
\begin{align} \label{eq:def P}
D_A P =A
\qquad \mbox{for any } A \in TM. 
\end{align}
\end{itemize}

Then we have $\alpha \neq 0,1$. 
Moreover, the pairs  
$(Ddf/(\alpha -1), f)$ 
and 
$(-f Dd \log f, f)$
are homogeneous pairs of degree $\alpha$. 
In particular, we can apply Theorem \ref{thm:split} and 
we have isometries 
\begin{align} \label{eq:decomp f log f}
\begin{split}
(M, (v \circ f) D d f) \cong & 
\left(\R_{>0} \times M_l, \frac{\alpha -1}{\alpha} \cdot \frac{v(r)}{r} dr^2 + \frac{r v(r)}{l} h_l \right), \\
(M, - (v \circ f) D d (\log f)) \cong & 
\left(\R_{>0} \times M_l, \frac{1}{\alpha} \cdot \frac{v(r)}{r^2} dr^2 - \frac{v(r)}{l} h_l \right)
\end{split}
\end{align}
for any function $v:\R_{>0} \rightarrow \R_{>0}.$
Here, $h_l$ is the induced pseudo-Riemannian metric on $M_l=f^{-1}(l) \subset M$ from $h=Ddf$. 
\end{proposition}

\begin{remark} \label{rem:hat g Hess}
If we set $g = Ddf/(\alpha -1)$, 
the equation (\ref{eq:dif log f}) implies that 
$\hat g = -f Dd \log f$, where $\hat g$ is defined in (\ref{eq:def hat g}). 
In particular, we can apply Corollaries \ref{cor:conf csc Hessian} and \ref{cor:bound log flat} 
to $-f Dd \log f$. 
\end{remark}

\begin{remark}
We can also prove the similar splitting for a pseudo-Riemannian metric 
$(v \circ f) D d (u \circ f)$ for some $u: \R_{>0} \rightarrow \R$, 
though $(D d (u \circ f), f)$ is not a homogeneous pair in general.  
That is,

\begin{enumerate}
\renewcommand{\labelenumi}{(\arabic{enumi})}
\item 
if 
$
\frac{d u}{dr} (r) \neq 0$ and $\frac{d \tilde{u}}{dr}  (r) \neq 0, 
$
where we set $\tilde{u}(r) = \alpha r u'(r) - u(r)$, 
$D d (u \circ f)$ is a pseudo-Riemannian metric on $M$. 

\item 
The map (\ref{eq:psi}) gives  
an isometry between 
$\left(\R_{>0} \times M_l, \frac{\tilde{u}'(r) v(r)}{\alpha r} dr^2 + \frac{r u'(r) v(r)}{l} h_l \right)$
and $\left(M, (v \circ f) D d (u \circ f) \right)$.
\end{enumerate}
However, since we do not know examples other than $u(r)=r$ or $\log r$, 
we omit the proof. 
We can prove this in the same way as Theorem \ref{thm:split}. 
\end{remark}

\begin{remark} \label{rem:TotaroLoftin}
Proposition \ref{prop:split met} generalizes 
\cite[Theorem 1]{Loftin} and \cite[Lemmas 2.1 and 2.4]{Totaro}. 
Indeed, (\ref{eq:invariant}) and (\ref{eq:def P}) are satisfied 
when $M \subset \R^n$ is a cone, 
$D$ is the canonical flat connection, and the $\R_{>0}$-action is the canonical one.  

By setting $v(r)=1, l=1$ and $r=s^2$ in 
(\ref{eq:decomp f log f}), we see that 
(M, Ddf) is isometric to 
$\left(\R_{>0} \times M_1, \frac{4 (\alpha -1)}{\alpha} ds^2 + s^2 h_1 \right)$, 
which is \cite[Lemma 2.1]{Totaro}. 

Similarly, 
by setting $v(r)=1, l=1, \alpha >0$ and $r= e^{\sqrt{\alpha} t}$ in 
(\ref{eq:decomp f log f}), we see that 
$(M, - D d (\log f))$ is isometric to 
$(\R \times M_1, dt^2 + (-h_1))$. 
This is \cite[Lemma 2.4]{Totaro}, which is equivalent to \cite[Theorem 1]{Loftin}. 
\end{remark}

\begin{proof}[Proof of Proposition \ref{prop:split met}]
Since the equation (\ref{eq:homog}) is the same as (\ref{eq:homog f}), 
we have $df (P) = \alpha f$ by (\ref{eq:df P}). 
Then 
differentiating $df (P) = \alpha f$, it follows that  
\begin{align*}
d (df(P)) = (D df)(P) + df (D P) 
\stackrel{(\ref{eq:def P})}
= (Ddf)(P) + df = \alpha df.
\end{align*}
Hence we have 
\begin{align} \label{eq:Hess pA pp}
h (P, \cdot) = (\alpha -1) d f, \qquad 
h (P, P) = \alpha (\alpha -1) f. 
\end{align}
Thus we see that $\alpha \neq 0,1$ so that $h=Ddf$ is a pseudo-Riemannian metric. 
Since the $\R_{>0}$-action preserves the connection $D$, 
we have 
$$
D_{(m_{\lambda^{-1}})_* A} \left(m_\lambda^* \alpha \right) = m_\lambda^* (D_A \alpha) 
$$
for $\lambda \in \R_{>0}$, a vector field $A \in \mathfrak{X}(M)$ and a 1-form $\alpha \in \Omega^1(M).$
Replacing $A$ with $(m_\lambda)_* A$, we see that 
$D (m_\lambda^* \alpha) = m_\lambda^* (D \alpha)$. 
Thus we obtain 
\begin{align}\label{eq:pullback h}
m_\lambda^* h = m_\lambda^* (Ddf) = D (m_\lambda^* d f) 
\stackrel{(\ref{eq:homog})} 
= \lambda^\alpha Ddf = \lambda^\alpha h. 
\end{align}
Hence (\ref{eq:pullback h}), (\ref{eq:homog}) and (\ref{eq:Hess pA pp}) 
imply that 
$(h/(\alpha-1), f) = (Ddf/(\alpha-1), f)$ 
is a homogeneous pair of degree $\alpha$. 
Then by Theorem \ref{thm:split}, we have an isometry 
$$
\left(M, \frac{(v \circ f) Ddf}{\alpha-1} \right) 
\cong 
\left( \R_{>0} \times M_l, 
v(r) \left( \frac{1}{\alpha r} dr^2 + \frac{r}{l} \cdot \frac{h_l}{\alpha-1} \right) \right), 
$$
which implies the first equation of (\ref{eq:decomp f log f}). 

Similarly, we have 
$$
m_\lambda^* Dd (\log f) = D (m_\lambda^* d (\log f) ) = Dd (\log f). 
$$
Since 
\begin{align} \label{eq:dif log f}
Dd (\log f)= D \left( \frac{df}{f} \right) = \frac{Ddf}{f} - \frac{df \otimes df}{f^2}, 
\end{align}
it follows that 
$$
Dd (\log f) (P, \cdot) = \frac{(\alpha -1) d f}{f} - \frac{\alpha df}{f} = - \frac{df}{f}
$$
by the equation $d f (P) = \alpha f$ and (\ref{eq:Hess pA pp}). 
Then we see that 
$(-f Dd (\log f), f)$ 
is a homogeneous pair of degree $\alpha$. 
Since $(- f D d (\log f))|_{M_l} = - h_l$ by (\ref{eq:dif log f}), 
Theorem \ref{thm:split} implies an isometry
$$
\left(M, - (v \circ f) f D d (\log f) \right) 
\cong 
\left( \R_{>0} \times M_l, 
v(r) \left( \frac{1}{\alpha r} dr^2 
+ \frac{r}{l} (- h_l) \right) \right), 
$$
Then replacing $v(r)$ with $v(r)/r$, 
we obtain the second equation of (\ref{eq:decomp f log f}). 
\end{proof}

\section{Examples} \label{sec:app}

In this section, we give examples 
to which we can apply results obtained in previous sections.

\subsection{Manifolds with flat Hessian metrics}

In this subsection, we give examples of 
manifolds with flat Hessian metrics. 
We can apply (1)--(4) in Section \ref{sec:intro} to these examples. 

\subsubsection{Cones in $\R^n$}

Many flat Hessian metrics are constructed on cones in $\R^n$. 
Let $D$ be the standard flat connection on $\R^n$.  
It is easy to see that $D$ satisfies (\ref{eq:invariant}) and (\ref{eq:def P}) 
w.r.t. the canonical $\R_{>0}$-action. 
We give examples of a function $f: \R^n \rightarrow \R$ such that 
$Ddf$ is flat on a cone in $\R^n$ 
where the Hessian of $f$ is positive definite. 

\begin{itemize}
\item
$f(x_1, \cdots, x_n) = x_1^2 + \cdots + x_n^2$, 
\item
$
f(x_1, \cdots, x_n) = f_1 (x_1,x_2) + f_2 (x_3,x_4)+ \cdots, 
$
where $f_1, f_2, \cdots$ are homogeneous functions  of two variables of the same degree
such that the Hessian matrices are positive definite. 
\item
$n=3$ and 
$
f(x,y,z)= x^6+y^6+z^6- 10 (x^3y^3+y^3z^3 + z^3x^3), 
$
which is called the Maschke sextic. 
\end{itemize}
The first $f$ is the most standard example. 
The flatness of $Ddf$ for the second $f$ is first proved by \cite[Section 6]{Totaro}, 
which also follows from Remark \ref{rem:2dim decomp}. 
That for the third $f$ is proved by \cite[Corollary 5.9 and Example 3]{Dubrovin}.


\subsection{Manifolds with pseudo-Hessian metrics}

In this subsection, we give examples of 
manifolds with pseudo-Hessian metrics. 
We can apply (1), (3) and (4) in Section \ref{sec:intro} to these examples.

\subsubsection{Regular convex cones}
An open convex cone $\Omega \subset \R^n$ not containing full straight lines 
is called a regular convex cone. 
The study of regular convex cones is considered to be the origin of the geometry 
of Hessian structures (\cite[Section 4]{Shima}). 
Let $(\R^n)^*$ be the dual space of $\R^n$. 
The dual cone $\Omega^* \subset (\R^n)^*$ is defined by 
$$
\Omega^* = \left\{ y^* \in (\R^n)^* \mid \la x, y^* \ra > 0 
\mbox{ for any } x \in \bar \Omega - \{ 0 \} \right \}, 
$$
where $\la \cdot, \cdot \ra$ is the canonical pairing of $\R^n$ and $(\R^n)^*$ 
and $\bar \Omega$ is the closure of $\Omega$. 
The characteristic function $f: \Omega \rightarrow \R$ is given by 
$$
f(x) = \int_{\Omega^*} e^{- \la x, x^* \ra} dx^*, 
$$
where $dx^*$ is the canonical volume form on $(\R^n)^*$. 
This function $f$ is homogeneous of degree $-n$, which follows \cite[(4.2)]{Shima}, 
and $Dd (\log f)$ defines a Riemannian metric on $\Omega$, 
where $D$ is the standard flat connection on the Euclidean space
(\cite[Proposition 4.5]{Shima}).

\subsubsection{The K\"ahler cone}

Let $M$ be a compact K\"ahler manifold of $\dim_{\C} M = n$. 
Let 
$$
\Kk = \{ \omega \in H^{1,1} (M, \R) \mid \omega \mbox{ contains a K\"ahler metric} \}
$$
be the K\"ahler cone of $M$, which is an open cone in $H^{1,1} (M, \R)$. 
Define a function $f:\Kk \rightarrow \R$ by 
$$
f(\omega) ={\rm Vol}(\omega) = \int_M \frac{\omega^n}{n!}, 
$$
which is homogeneous of degree $n$
w.r.t. the canonical $\R_{>0}$-action on $\Kk$. 
Then it is known that $g=- Dd (\log f)$ is a Riemannian metric on $\Kk$, 
where $D$ is the standard flat connection on $\Kk$
(\cite[Proposition 1.1]{Mag}). 
The Riemannian metric $g$ is complete if and only if
$\Kk$ is a connected component of the volume cone 
$\{ \omega \in H^{1,1} (M, \R) \mid \int_M \omega^n >0 \}$ 
(\cite[Proposition 4.4]{Mag}).

The level sets $\Kk_l = f^{-1}(l) \subset \Kk$, where $l >0$, 
with the induced Riemannian metric $g_l$ 
was studied in \cite{Huybrechts, Wilson, TW}.  
Wilson explicitly computed the curvature tensor and the geodesics of $g_l$. 
He conjectured that when $M$ is a Calabi-Yau manifold, 
$\Kk_l$ should have non-positive sectional curvatures, 
at least in the large K\"ahler structure limit, 
considering the correspondence to the Weil-Petersson metric 
on the complex moduli space under mirror symmetry. 
Now, there are some counterexamples in \cite{Totaro, TW}. 

When $h^{1,1} = \dim H^{1,1}(X, \R) =2$ or $M$ is hyperk\"ahler, 
$g_l$ has constant negative sectional curvature. 
See \cite[p.631 and Example 1]{Wilson}.

\subsubsection{The $G_2$ moduli space} \label{subsec:G2 moduli}

The exceptional Lie group $G_2$ is realized as a stabilizer in ${\rm GL} (7, \R)$ of a 3-form 
$\varphi_0$ on $\R^7$. 
The ${\rm GL}_+(7, \R)$-orbit ${\rm GL}_+(7, \R) \cdot \varphi_0$ through $\varphi_0$, 
where  ${\rm GL}_+(7, \R) = \{ A \in {\rm GL} (7, \R) \mid \det A >0 \}$, 
is diffeomorphic to ${\rm GL}_+(7, \R)/G_2$. 
It has the same dimension as $\Lambda^3 (\R^7)^*$, 
and hence it is open in $\Lambda^3 (\R^7)^*$. 
Any $\varphi \in {\rm GL}_+(7, \R) \cdot \varphi_0$ 
induces the metric $g_\varphi$, the volume form ${\rm vol}_\varphi$ 
and the Hodge star $*_\varphi$ on $\R^7$. 
They are related by 
$$
i(u) \varphi \wedge i(v) \varphi \wedge \varphi
= 
6 g_\varphi (u,v) {\rm vol}_\varphi, \qquad
\varphi \wedge *_\varphi  \varphi = 7 {\rm vol}_\varphi 
\qquad \mbox{for } u,v \in \R^7. 
$$

Let $M^7$ be a 7-dimensional manifold with a $G_2$-structure. 
That is, the tangent frame bundle is reduced to a $G_2$-subbundle. 
We assume that $M^7$ is connected for simplicity. 
We can define a positive 3-form, a section of an open subbundle $\Lambda^3_+ T^*M^7$ 
of $\Lambda^3 T^*M^7$, which is induced from ${\rm GL}_+(7, \R) \cdot \varphi_0$. 
We denote by $\nabla^\varphi$ the Levi-Civita connection of $g_\varphi$. 
Then a Riemannian metric $g$ has holonomy contained in $G_2$ 
if and only if there exists a positive 3-form 
such that $\nabla^\varphi \varphi =0$ and $g =g_\varphi$. 
A positive 3-form $\varphi$ satisfying $\nabla^\varphi \varphi =0$ 
is called a torsion-free $G_2$-structure. 

The holonomy group of $g_\varphi$ for a torsion-free $G_2$-structure $\varphi$ 
is determined by the topology of $M^7$. 
It has full holonomy $G_2$ if and only if $\pi_1 (M^7)$ is finite. 
We call such a manifold irreducible.

Define the moduli space $\Mm_{G_2}$ of torsion-free $G_2$-structures by 
$$
\Mm_{G_2} = \{ \varphi \in \Omega^3_+(M^7) \mid \nabla^\varphi \varphi = 0 \}/ {\rm Diff}_0 (M^7),  
$$
where $\Omega^3_+(M^7)$ is the space of smooth positive 3-forms and 
${\rm Diff}_0 (M^7)$ is the identity component of the diffeomorphism group. 

Suppose that $M^7$ is compact. 
By \cite{Joyce2}, a map 
$\Mm_{G_2} \ni [\varphi] \mapsto [\varphi] \in H^3 (M^7, \R)$ 
is a local homeomorphism, 
which implies that 
$\Mm_{G_2}$ is an affine manifold of dimension $b^3 = \dim H^3 (M^7, \R)$. 
Denote by $D$ the flat connection on $\Mm_{G_2}$ 
(cf. \cite[Section 3.1]{KL}). 
This satisfies (\ref{eq:invariant}) and (\ref{eq:def P}) 
w.r.t. the canonical $\R_{>0}$-action on $\Mm_{G_2}$.  
Define $f: \Mm_{G_2} \rightarrow \R$ by 
$$
f ([\varphi]) = 3 {\rm Vol}(\varphi) = 3 \int_{M^7} {\rm vol}_\varphi 
= \frac{3}{7} \int_{M^7} \varphi \wedge *_\varphi \varphi, 
$$
which is homogeneous of degree $7/3$ 
w.r.t. the canonical $\R_{>0}$-action on $\Mm_{G_2}$. 
We have three canonical pseudo-Riemannian metrics on $\Mm_{G_2}$
(cf. \cite[Proposition 22]{Hitchin2}, \cite[Theorem 3.5 and Lemma 3.11]{KL}).

\begin{enumerate}
\renewcommand{\labelenumi}{(\arabic{enumi})}
\item The tensor $h_1=Ddf$ is a pseudo-Riemannian metric with signature $(1 + b^1, b^3-1-b^1)$, 
where $b^i$ is the $i$-th Betti number.

\item The tensor $h_2 =- Dd (\log f)$ is a pseudo-Riemannian metric with signature $(b^3-b^1, b^1)$.
When $M$ is irreducible, this is positive definite. 

\item 
By identifying $T_{[\varphi]} \Mm_{G_2}$ with  $\Hh^3_{\varphi}$, 
the space of harmonic 3-forms w.r.t. $g_\varphi$, 
the $L^2$-metric on $M$ induces the metric $g_{L^2}$ on $\Mm_{G_2}$. 
When $M$ is irreducible, we have 
$$
g_{L^2} = -f Dd \log f.
$$
\end{enumerate}

By Proposition \ref{prop:split met}, $(\frac{3}{4} h_1, f)$ and $(g_{L^2}, f)$ 
are homogeneous pairs on $\Mm_{G_2}$. 
These two Riemannian-metrics are related by 
$g_{L^2} =  \widehat{\frac{3}{4} h_1}$, by Remark \ref{rem:hat g Hess}. 
By the conformal transformation of $g_{L^2}$, we obtain $h_2$.

\begin{remark}
As far as the author knows, 
there are no known examples of a 7-dimensional manifold 
admitting a torsion-free $G_2$-structure with $b^3=\dim \Mm_{G_2} = 2$. 
It would be interesting to construct such examples. 
It is because the above pseudo-Riemannian metrics are flat by Corollary \ref{cor:2dim csc Hess}, 
and hence $\Mm_{G_2}$ is expected to have simpler geometric structures, 
which might be useful to study the general cases. 

The Hessian curvature tensor for $h_2$ is explicitly given in \cite{GY}. 
See also \cite{Grigorian}. 
The detailed analysis of the curvature of $\Mm_{G_2}$ is also given in \cite{KLL}. 
\end{remark}

The metric completion of $\Mm_{G_2}$ has not been studied yet. 
By Theorem \ref{thm:comp}, we see the following. 

\begin{corollary} \label{cor:comp G2}
When $M^7$ is irreducible, the metric completion $\overline{(\Mm_{G_2}, d_{h_2})}$ of $\Mm_{G_2}$ 
(resp. $\overline{(\Mm_{G_2}, d_{g_{L^2}})}$) w.r.t. the metric 
$d_{h_2}$ (resp. $d_{g_{L^2}}$) induced from $h_2$ (resp. $g_{L^2}$) 
is homeomorphic to 
$\R_{>0} \times \overline{(\Mm_{G_2})_l}$ 
(resp. $(\{ 0 \} \cup \R_{> 0}) \times \overline{(\Mm_{G_2})_l} / \left( \{ 0 \} \times \overline{M_l} \right)$), 
where $\overline{(\Mm_{G_2})_l}$ 
is the metric completion of $(\Mm_{G_2})_l = f^{-1}(l) \subset \Mm_{G_2}$ 
w.r.t. the induced Riemannian metric from $h_2$. 

In particular, $\overline{(\Mm_{G_2}, d_{h_2})}$ is strictly smaller than 
$\overline{(\Mm_{G_2}, d_{g_{L^2}})}$. 
In other words, $\overline{(\Mm_{G_2}, d_{h_2})}$ 
has less degenerate points than $\overline{(\Mm_{G_2}, d_{g_{L^2}})}$. 
\end{corollary}

\begin{remark} \label{rem:comp G2}
The completion of the space of Riemannian metrics 
is described geometrically in terms of 
measurable, symmetric, positive semidefinite 
$(0, 2)$-tensor fields (cf. \cite{Clarke3, CR1}, Section \ref{subsec:Riem met}). 
We may also expect to describe $\overline{(\Mm_{G_2})_l}$ geometrically, 
which is equivalent to describe  
$\overline{(\Mm_{G_2}, d_{h_2})}$ or $\overline{(\Mm_{G_2}, d_{g_{L^2}})}$ geometrically 
by Theorem  \ref{thm:comp} and Proposition \ref{prop:comp level}, 
but it seems to be difficult. 

If we follow the case above, 
we have to 
describe geometrically the metric completions of $\Omega^3_+(M^7)$ 
and $\widehat{\Mm}_{G_2} = \{ \varphi \in \Omega^3_+(M^7) \mid \nabla^\varphi \varphi = 0 \}$ 
to give a geometrical description of $\Mm_{G_2} = \widehat{\Mm}_{G_2}/{\rm Diff}_0 (M^7)$. 

We expect that the metric completion of $\Omega^3_+(M^7)$ 
w.r.t. the $L^2$ metric divided by the volume functional 
(which corresponds to $h_2$ on $\Mm_{G_2}$) 
will be 
the set of measurable semi-positive $G_2$-structures 
with nonzero finite volume 
modulo the equivalence relation as in Theorem \ref{thm:Clarke3}. 
Here, we call the section of the closure of $\Lambda^3_+ T^*M^7$ ``semi-positive". 

It will be difficult to determine the metric completion of 
$\widehat{\Mm}_{G_2} = \{ \varphi \in \Omega^3_+(M^7) \mid \nabla^\varphi \varphi = 0 \}$.  
The related problem is considered in \cite{CR1}, 
where the Calabi-Yau theorem is used in the proof. 
There is no such analogues in the $G_2$ case. 
For the metric completion of $\Mm_{G_2}$, 
there will be more problems 
when taking the quotient by ${\rm Diff}_0 (M^7)$ as in \cite[Section 5]{Clarke3}. 

It would also be an interesting question to 
study the metric completion of the space of closed $G_2$-structures 
$\{ \varphi \in \Omega^3_+(M^7) \mid d \varphi = 0 \}$. 
We might characterize the 
existence of torsion-free $G_2$-structures 
in terms of the ``analytic stability condition" 
in terms of the Laplacian flow and the metric completion as in \cite{CR1}. 
\end{remark}

\subsubsection{The ${\rm SL}(3, \C)$ moduli space}

The group ${\rm SL}(3, \C)$ is realized as a stabilizer in ${\rm GL} (6, \R)$ of a 3-form 
$\psi_0 = {\rm Re}(dz^1 \wedge dz^2 \wedge dz^3)$ on $\R^6$, 
where we use holomorphic coordinates $(z^1, z^2, z^3)$ on $\C^3 \cong \R^6$. 
The ${\rm GL}_+(6, \R)$-orbit ${\rm GL}_+(6, \R) \cdot \psi_0$ through $\psi_0$, 
where  ${\rm GL}_+(6, \R) = \{ A \in {\rm GL} (6, \R) \mid \det A >0 \}$, 
is diffeomorphic to ${\rm GL}_+(6, \R)/{\rm SL}(3, \C)$. 
It has the same dimension as $\Lambda^3 (\R^6)^*$, 
and hence it is open in $\Lambda^3 (\R^6)^*$. 
By \cite[(9), (10)]{Hitchin2}, 
Any $\psi \in {\rm GL}_+(6, \R) \cdot \psi_0$ 
induces a complex structure $J_\psi$  
and a 3-form $\hat \psi$ on $\R^6$ such that 
$\psi + i \hat \psi$ is a $(3,0)$-form w.r.t. $J_\psi$. 

Let $M^6$ be a 6-dimensional manifold with a ${\rm SL}(3, \C)$-structure. 
That is, the tangent frame bundle is reduced to a ${\rm SL}(3, \C)$-subbundle. 
We can define a positive 3-form, a section of an open subbundle $\Lambda^3_+ T^*M^6$ 
of $\Lambda^3 T^*M^6$, 
which is induced from ${\rm GL}_+(6, \R) \cdot \psi_0$. 
We call a positive 3-form $\psi$ torsion-free 
if $d \psi = d \hat \psi =0$.

Define the moduli space $\Mm_{{\rm SL}(3, \C)}$ of 
torsion-free ${\rm SL}(3, \C)$-structures by 
$$
\Mm_{{\rm SL}(3, \C)} 
= \{ \psi \in \Omega^3_+(M^6) \mid d \psi = d \hat \psi = 0 \}/ {\rm Diff}_0 (M^6),  
$$
where $\Omega^3_+(M^6)$ is the space of smooth positive 3-forms and 
${\rm Diff}_0 (M^6)$ is the identity component of the diffeomorphism group. 

Suppose that 
$M^6$ is a compact complex 3-manifold with non-vanishing holomorphic
3-form and satisfy the $\p \bar \p$-lemma (such as a Calabi-Yau manifold). 
Then by \cite[Section 6.3]{Hitchin2}, a map 
$\Mm_{{\rm SL}(3, \C)} \ni [\psi] \mapsto [\psi] \in H^3 (M^6, \R)$ 
is a local homeomorphism, 
which implies that 
$\Mm_{{\rm SL}(3, \C)}$ is an affine manifold of dimension $b^3(M^6)$. 
Denote by $D$ the flat connection on $\Mm_{{\rm SL}(3, \C)}$. 
This satisfies (\ref{eq:invariant}) and (\ref{eq:def P}) 
w.r.t. the canonical $\R_{>0}$-action on $\Mm_{{\rm SL}(3, \C)}$.  
Define $f: \Mm_{{\rm SL}(3, \C)} \rightarrow \R$ by 
$$
f ([\psi]) = \int_{M^6} \psi \wedge \hat \psi
$$
which is homogeneous of degree $2$ 
w.r.t. the canonical $\R_{>0}$-action on $\Mm_{G_2}$
by the definition of $\hat \psi$ in \cite[Definition 2]{Hitchin2}. 
Then the Hessian $Ddf$ of $f$ defines a pseudometric on $\Mm_{{\rm SL}(3, \C)}$. 
In fact, $\Mm_{{\rm SL}(3, \C)}$ admits a more geometric structure. 
It is known to be a special pseudo-K\"ahler manifold (\cite[Proposition 17]{Hitchin2}).

\subsection{Other examples}
In this subsection, we give examples which admit a homogeneous pair 
but are not known to admit pseudo-Hessian structures. 
We can also apply (1), (3) and (4) in Section \ref{sec:intro} to these examples.

\subsubsection{The ${\rm Spin}(7)$ moduli space}

The group ${\rm Spin}(7)$ is realized as a stabilizer in ${\rm GL} (8, \R)$ 
of a 4-form $\Phi_0$ on $W=\R^8$. 
It is known that ${\rm Spin}(7) \subset {\rm SO}(8)$. 
The ${\rm GL}_+(8, \R)$-orbit ${\rm GL}_+(8, \R) \cdot \Phi_0$ through $\Phi_0$, 
where  ${\rm GL}_+(8, \R) = \{ A \in {\rm GL} (8, \R) \mid \det A >0 \}$, 
is diffeomorphic to ${\rm GL}_+(8, \R)/{\rm Spin}(7)$. 
Note that this is not open in $\Lambda^4 W^*$ as in the cases $G_2$ and ${\rm SL}(3, \C)$. 
Any $\Phi \in {\rm GL}_+(8, \R) \cdot \Phi_0$ 
induces the metric $g_\Phi$, the volume form ${\rm vol}_\Phi$ 
and the Hodge star $*_\Phi$ on $\R^8$.
Note that $\Phi$ and ${\rm vol}_\Phi$ are related by 
$$
\Phi \wedge \Phi = 14 {\rm vol}_\Phi. 
$$
The group ${\rm Spin}(7)$ acts canonically on the space of forms $\Lambda^* W^*$ on $W$. 
In particular, $\Lambda^4 W^*$ has the following irreducible decomposition 
\begin{align} \label{eq:Spin7 irr decomp}
\Lambda^4 W^* = \Lambda^4_1 W^* \oplus \Lambda^4_7 W^* 
\oplus \Lambda^4_{27} W^* \oplus \Lambda^4_{35} W^*, 
\end{align}
where $\Lambda^4_k W^*$ is a $k$-dimensional irreducible representation of ${\rm Spin}(7)$. 
Note that $\Lambda^4_1 W^* = \R \Phi_0$ and 
\begin{align} \label{eq:Spin7 irr decomp 4}
\Lambda^4_1 W^* \oplus \Lambda^4_7 W^* \oplus \Lambda^4_{27} W^* = \Lambda^4_+ W^*, 
\qquad
\Lambda^4_{35} W^* = \Lambda^4_- W^*, 
\end{align}
where $\Lambda^4_+ W^*$ (resp. $\Lambda^4_- W^*$) is the space of 
self-dual (resp. anti-self-dual) 4-forms.

Let $M^8$ be an 8-dimensional manifold with a ${\rm Spin}(7)$-structure, 
that is, the tangent frame bundle is reduced to a ${\rm Spin}(7)$-subbundle. 
We assume that $M^8$ is connected for simplicity. 
We can define an admissible 4-form, a section of a 
$43(=1+7+35)$-dimensional subbundle $\Aa^4 M^8$ of $\Lambda^4 T^*M^8$, 
which is induced from ${\rm GL}_+(8, \R) \cdot \Phi_0$. 
We denote by $\nabla^\Phi$ the Levi-Civita connection of $g_\Phi$. 
Then a Riemannian metric $g$ has holonomy contained in ${\rm Spin}(7)$ 
if and only if there exists an admissible 4-form $\Phi$ 
such that $\nabla^\Phi \Phi =0$ and $g =g_\Phi$. 
It is known that 
$\nabla^\Phi \Phi =0$ if and only if $d \Phi=0$. 
An admissible 4-form $\Phi$ satisfying $d \Phi =0$ 
is called a torsion-free ${\rm Spin}(7)$-structure. 

The holonomy group of $g_\Phi$ for a torsion-free ${\rm Spin}(7)$-structure $\Phi$ 
is determined by the topology of $M^8$. 
It has full holonomy ${\rm Spin}(7)$ if and only if 
$M^8$ is simply connected and 
the Betti numbers of $M^8$ satisfy 
$b^3+b^4_+=b^2+2b^4_- + 25$. 
We call such a manifold irreducible. 
In this case, we have $b^4_7=0$ (cf. \cite[Proposition 10.6.5 and Theorem 10.6.8]{Joyce3}).

On a manifold $M^8$ admitting a torsion-free ${\rm Spin}(7)$-structure $\Phi$, 
there is a decomposition of $\Omega^4(M^8)$, 
the space of smooth 4-forms, induced from (\ref{eq:Spin7 irr decomp}):
$$
\Omega^4 (M^8) = 
\Omega^4_1 (M^8) \oplus \Omega^4_7 (M^8) \oplus \Omega^4_{27} (M^8) \oplus \Omega^4_{35} (M^8), 
$$
where we denote by $\Omega^4_k (M^8)$ the space of smooth sections of 
$\Lambda^4_k T^* M^8$. 
Set $(\Hh^4_k)_\Phi = \{ \xi \in \Omega^4_k (M^8) \mid d \xi = d *_\Phi \xi = 0 \}$, 
which is the space of harmonic forms in $\Omega^4_k (M^8)$, 
$b^4_k = \dim (\Hh^4_k)_\Phi$ and  
\begin{align} \label{eq:Spin7 harm}
\Hh_\Phi = (\Hh^4_1)_{\Phi} \oplus (\Hh^4_7)_{\Phi} \oplus (\Hh^4_{35})_{\Phi}. 
\end{align}

Define the moduli space $\Mm_{{\rm Spin}(7)}$ of torsion-free ${\rm Spin}(7)$-structures by 
\begin{align*}
\widehat \Mm_{{\rm Spin}(7)} &= \{ \Phi \in C^\infty (\Aa^4 M^8) \mid d \Phi = 0 \}, \\
\Mm_{{\rm Spin}(7)} &= \widehat \Mm_{{\rm Spin}(7)}/ {\rm Diff}_0 (M^8),  
\end{align*}
where $C^\infty (\Aa^4 M^8)$ is the space of smooth admissible 4-forms and 
${\rm Diff}_0 (M^8)$ is the identity component of the diffeomorphism group. 
Let $\pi: \widehat \Mm_{{\rm Spin}(7)} \rightarrow \Mm_{{\rm Spin}(7)}$ be the canonical projection. \\

As far as the author knows, the geometric structures of $\Mm_{{\rm Spin}(7)}$ have not been studied yet. 
Thus by recalling the result of \cite{Joyce1} about the smoothness of $\Mm_{{\rm Spin}(7)}$, 
we explain two pseudo-Riemannian metrics on $\Mm_{{\rm Spin}(7)}$ in detail.

Suppose that $M^8$ is compact. 
By \cite{Joyce1}, 
by fixing any $\Phi \in \widehat \Mm_{{\rm Spin}(7)}$, 
there exist open neighborhoods 
$\Uu \subset \Hh_\Phi$ 
of $0$ and 
$\Vv \subset \Mm_{{\rm Spin}(7)}$ 
of $\pi(\Phi)$
and a smooth map $\widehat \Phi: \Uu \rightarrow \widehat \Mm_{{\rm Spin}(7)}$
such that 
$\widehat \Phi (0) = \Phi, (d \widehat \Phi)_0 (\xi) = \xi$ for any 
$\xi \in \Hh_\Phi$, 
and 
$
\pi \circ \widehat \Phi: \Uu \rightarrow \Vv
$
is a homeomorphism. 
Then we see that 
$\Mm_{{\rm Spin}(7)}$ is a smooth manifold of dimension 
$b^4_1 + b^4_7 + b^4_{35}$, 
which is known to be a topological invariant. 
Thus we have the identification 
\begin{align} \label{eq:Spin7 tangent}
T_{\pi(\Phi)} \Mm_{{\rm Spin}(7)} = 
(d \pi)_\Phi (\Hh_\Phi). 
\end{align}
However, $\Mm_{{\rm Spin}(7)}$ is not known to be an affine manifold as in the 
cases of $G_2$ and ${\rm SL}(3, \C)$.

By (\ref{eq:Spin7 tangent}), 
we can define two canonical pseudo-Riemannian metrics 
$g_{I}$ and $g_{L^2}$ on $\Mm_{{\rm Spin}(7)}$. 
\begin{enumerate}
\renewcommand{\labelenumi}{(\arabic{enumi})}
\item 
For $\Phi \in \widehat \Mm_{{\rm Spin}(7)}$ and $\xi, \eta \in \Hh_\Phi$, define 
\begin{align}\label{eq:def gI}
(g_I)_{\pi(\Phi)} \left((d \pi)_\Phi (\xi), (d \pi)_\Phi (\eta) \right) = \int_{M^8} \xi \wedge \eta, 
\end{align}
which is induced from the intersection form on $H^4 (M^8, \R)$. 

\item 
For $\Phi \in \widehat \Mm_{{\rm Spin}(7)}$ and $\xi, \eta \in \Hh_\Phi$, define 
\begin{align}\label{eq:def gL2}
(g_{L^2})_{\pi(\Phi)} \left((d \pi)_\Phi (\xi), (d \pi)_\Phi (\eta) \right) 
= \int_{M^8} \xi \wedge *_\Phi \eta,
\end{align}
which is induced from the $L^2$-metric on $M^8$, 
and hence $g_{L^2}$ is positive definite. 
\end{enumerate}

\begin{lemma}
The pseudo-Riemannian metrics $g_I$ and $g_{L^2}$ are well-defined. 
\end{lemma}

\begin{proof}
Take any $\Phi \in \widehat \Mm_{{\rm Spin}(7)}$ and $\theta \in {\rm Diff}_0 (M^8)$. 
The Riemannian metric $g_\Phi$ induced from $\Phi$ is 
given explicitly in \cite[Theorem 4.3.5]{Karigiannis}, which implies that 
$$
g_{\theta^* \Phi} = \theta^* g_\Phi. 
$$
Then we easily see that the induced Hodge stars are related by 
\begin{align} \label{eq:Spin7 Hodge relation}
*_{\theta^* \Phi} = \theta^* *_\Phi (\theta^{-1})^*. 
\end{align}
Then for any $\xi \in \Hh_\Phi$, 
we have  
$\theta^* \xi \in \Hh_{\theta^* \Phi}$. 
The equation $\pi = \pi \circ \theta^*$ implies that 
$(d \pi)_\Phi (\xi) = (d \pi)_{\theta^* \Phi} (\theta^* \xi)$. 
Thus we only have to prove 
$$
\int_{M^8} \theta^* \xi \wedge \theta^* \eta = \int_{M^8} \xi \wedge \eta, \qquad
\mbox{and} \qquad
\int_{M^8} \theta^* \xi \wedge *_{\theta^* \Phi} \theta^* \eta = \int_{M^8} \xi \wedge *_\Phi \eta 
$$
for any $\xi, \eta \in \Hh_\Phi$. 
These equations follow from $\theta \in {\rm Diff}_0 (M^8)$ and 
(\ref{eq:Spin7 Hodge relation}). 
\end{proof}

If we decompose 
$\xi = \xi_1 + \xi_7 + \xi_{35}$ and $\eta = \eta_1 + \eta_7 + \eta_{35}$
following (\ref{eq:Spin7 harm}), 
the equation (\ref{eq:Spin7 irr decomp 4}) implies that 
\begin{align} 
(g_I)_{\pi(\Phi)} \left((d \pi)_\Phi (\xi), (d \pi)_\Phi (\eta) \right)
=& 
\int_{M^8} \xi_1 \wedge *_\Phi \eta_1 + \int_{M^8} \xi_7 \wedge *_\Phi \eta_7 
- \int_{M^8} \xi_{35} \wedge *_\Phi \eta_{35},  
\label{eq:def gI 2} \\
(g_{L^2})_{\pi(\Phi)} \left((d \pi)_\Phi (\xi), (d \pi)_\Phi (\eta) \right)
=& 
\int_{M^8} \xi_1 \wedge \eta_1 + \int_{M^8} \xi_7 \wedge \eta_7 - \int_{M^8} \xi_{35} \wedge \eta_{35}. 
\label{eq:def gL2 2}
\end{align}
In particular, (\ref{eq:def gI 2}) implies that $g_I$ has signature $(1+b^4_7, b^4_{35})$ 
and it is Lorentzian  
if $M^8$ is irreducible.

Define a function $f: \Mm_{{\rm Spin}(7)} \rightarrow \R$ by 
$$
f (\pi(\Phi)) = 7 {\rm Vol}(\Phi) = 7 \int_{M^8} {\rm vol}_\Phi = \frac{1}{2} \int_{M^8} \Phi \wedge \Phi, 
$$
which is homogeneous of degree $2$ 
w.r.t. the canonical $\R_{>0}$-action on $\Mm_{{\rm Spin}(7)}$.

\begin{proposition} \label{prop:homog pair Spin7}
The pairs $(g_I, f)$ and $(g_{L^2}, f)$ are homogeneous pairs of degree $2$ 
w.r.t. the canonical $\R_{>0}$-action on $\Mm_{{\rm Spin}(7)}$. 
\end{proposition}
\begin{proof}
By (\ref{eq:def gI}) and (\ref{eq:def gL2 2}), we see that 
(\ref{eq:homog g}) and (\ref{eq:homog f}) are satisfied for $\alpha = 2$. 
The vector field $P$ generated by 
the canonical $\R_{>0}$-action on $\Mm_{{\rm Spin}(7)}$ 
is given by 
$$
P_{\pi(\Phi)} = (d \pi)_\Phi (\Phi).
$$
Then by (\ref{eq:def gI 2}), we have 
for any $\eta \in \Hh_\Phi$, 
$$
(g_I)_{\pi(\Phi)} (P_{\pi(\Phi)}, (d \pi)_\Phi (\eta)) 
= \int_{M^8} \Phi \wedge *_\Phi \eta_1 
= (g_{L^2})_{\pi(\Phi)} (P_{\pi(\Phi)}, (d \pi)_\Phi (\eta)). 
$$
On the other hand, we compute 
\begin{align*}
(df)_{\pi(\Phi)} ((d \pi)_\Phi (\eta)) 
&= d (f \circ  \pi \circ \widehat \Phi)_0 \left( \left. \frac{d}{dt} (t \eta) \right|_{t=0} \right)\\
&= \frac{1}{2} \left. \frac{d}{dt} \int_{M^8} \widehat \Phi (t \eta) \wedge \widehat \Phi (t \eta) \right|_{t=0}\\
&= \int_{M^8} \eta \wedge \Phi
= (g_I)_{\pi(\Phi)} (P_{\pi(\Phi)}, (d \pi)_\Phi (\eta)). 
\end{align*}
Hence the proof is completed. 
\end{proof}

Two pseudo-Riemannian metrics $g_I$ and $g_{L^2}$ are related as follows. 

\begin{remark}
If $M^8$ is irreducible, $g_I$ and $g_{L^2}$ are related by 
$g_{L^2} = \widehat{g_I}$, where we use the notation in (\ref{eq:def hat g}). 

Indeed, 
take any $\Phi \in \widehat \Mm_{{\rm Spin}(7)}$ and $\eta = \eta_1 + \eta_{35} 
\in (\Hh^4_1)_\Phi \oplus (\Hh^4_{35})_\Phi$. 
By the proof of Proposition \ref{prop:homog pair Spin7}, we have 
$(df)_{\pi(\Phi)} ((d \pi)_\Phi (\eta)) = \la \eta, \Phi \ra_{L^2}$, 
where $\la \cdot, \cdot \ra$ is the $L^2$-metric on the space of differential forms on $M^8$ 
induced from $g_{\Phi}$. 
Since $(\Hh^4_1)_\Phi = \R \Phi$, we have 
$\eta_1 = \la \eta, \Phi \ra_{L^2} \Phi/ \la \Phi, \Phi \ra_{L^2}$. 
Then we compute 
$$
\frac{(df \otimes df)_{\pi(\Phi)} ((d \pi)_\Phi (\eta), (d \pi)_\Phi (\eta)) }{f (\pi (\Phi))}
= 
\frac{2 \la \eta, \Phi \ra_{L^2}^2}{\la \Phi, \Phi \ra_{L^2}}
=
2 \la \eta, \eta_1 \ra_{L^2}
=
2 \la \eta_1, \eta_1 \ra_{L^2}. 
$$
Then (\ref{eq:def hat g}) and (\ref{eq:def gI 2}) imply that 
$$
(\widehat{g_I})_{\pi(\Phi)} \left((d \pi)_\Phi (\eta), (d \pi)_\Phi (\eta) \right)
= 
2 \la \eta_1, \eta_1 \ra_{L^2} 
- \left( \la \eta_1, \eta_1 \ra_{L^2} - \la \eta_{35}, \eta_{35} \ra_{L^2} \right)
= 
\la \eta, \eta \ra_{L^2}. 
$$

\end{remark}

\begin{remark}
As far as the author knows, 
there are no known examples of an 8-dimensional manifold 
admitting a torsion-free ${\rm Spin}(7)$-structure with $\dim \Mm_{{\rm Spin}(7)}=2$. 
It would be interesting to construct such examples. 
It is because the above pseudo-Riemannian metrics are flat by Corollary \ref{cor:2dim csc Hess}, 
and hence $\Mm_{{\rm Spin}(7)}$ is expected to have simpler geometric structures, 
which might be useful to study the general cases.

The metric completion of $\Mm_{{\rm Spin}(7)}$ has not been studied yet. 
We expect that the same statements as in Remark \ref{rem:comp G2} hold. 
\end{remark}



\subsubsection{The space of Riemannian metrics}
\label{subsec:Riem met}

Let $M$ be a compact oriented $n$-dimensional manifold 
and let $\Mm$ be the space of all smooth Riemannian metrics on $M$. 
This is an open cone in the Fr\'echet space $\Gamma(S^2 T^*M)$, 
the space of symmetric $(0, 2)$-tensors on $M$. 
Thus $\Mm$ is a Fr\'echet manifold 
and its tangent space at $g \in \Mm$ is canonically identified with $\Gamma(S^2 T^*M)$. 

For $g \in \Mm$ and $h_1, h_2 \in T_g \Mm \cong \Gamma(S^2 T^*M)$, 
define a weak Riemannian metric $g_E$, which is called the Ebin metric, on $\Mm$ by 
\begin{align} \label{eq:Ebin metric}
(g_E)_g (h_1, h_2) = \int_M {\rm tr} (g^{-1} h_1 g^{-1} h_2) {\rm vol}_g, 
\end{align}
where $g^{-1} h_i \in \Gamma(T^*M \otimes TM)$ is 
the contraction of the dual Riemannian metric of $g$ and $h_i$, 
and ${\rm vol}_g$ is the volume form induced from $g$. 

The local structure of $(\Mm, g_E)$ was first studied in \cite{FG, GM}. 
The authors first proved the splitting similar to Theorem \ref{thm:split} for $(\Mm, g_E)$. 
Then they showed that the sectional curvature of $g_E$ is 
nonpositive (\cite[Corollary 1.17]{FG}) 
and gave the geodesics explicitly (\cite[Theorem 2.3]{FG}, \cite[Theorem 3.2]{GM}).

On $\Mm$, there is a canonical function $f: \Mm \rightarrow \R$ given by 
\begin{align} \label{eq:Ebin vol}
f(g) = 2 {\rm Vol}(g) = 2 \int_{M} {\rm vol}_g. 
\end{align}

Clarke showed that 
the pseudometric $d_{g_E}$ induced from $g_E$ is the metric in \cite{Clarke1} 
and determined the metric completion $\overline{(\Mm, d_{g_E})}$ of $\Mm$ w.r.t. $d_{g_E}$ 
in \cite{Clarke3}. 

\begin{theorem}[{\cite[Theorem 5.19]{Clarke3}}] \label{thm:Clarke3}
Let $\Mm_{finite}$ be the set of measurable positive-semidefinite sections $g :M \rightarrow S^2T^*M$ 
with $f(g) < \infty$. 
Set $\widehat{\Mm_{finite}} = \Mm_{finite}/ \sim$, 
where $\sim$ is the equivalence relation defined by 
$g \sim h \Leftrightarrow$ 
$g(x)=h(x)$ or $g(x) \neq h(x)$ and $\det g(x) = \det h(x)=0$ 
for almost everywhere $x \in M$. 

Then the metric completion $\overline{(\Mm, d_{g_E})}$ 
of $\Mm$ w.r.t. $d_{g_E}$ is identified with $\widehat{\Mm_{finite}}$. 
\end{theorem}

For the proof, Clarke first introduced a notion of 
the $\omega$-convergence for Cauchy sequences in $\Mm$, 
which is a kind of pointwise a.e.-convergence. 
Then as summarized in \cite[p.60]{Clarke4}, 
Theorem \ref{thm:Clarke3} is proved in the following steps. 

(i) For any Cauchy sequence $\{g_k \} \subset \Mm$, 
there exists an $\omega$-convergent subsequence. 
Denote by $[g_0] \in \widehat{\Mm_{finite}}$ the $\omega$-limit. 
(ii) Two $\omega$-convergent subsequences 
$\{ g^0_k \}$ and $\{ g^1_k \}$ have the same $\omega$-limit 
if and only if $[g^0_k] = [g^1_k] \in \overline{(\Mm, d_{g_E})}.$
(iii) For each element of $\Mm_{finite}$, 
there exists a sequence in $\Mm$ $\omega$-converging to it. 

Then a map 
$\overline{(\Mm, d_{g_E})} \ni [g_k] \mapsto [g_0] \in \widehat{\Mm_{finite}}$
gives a bijection, 
where we use the notation in Definition \ref{def:comp}. 
Note that by \cite[Theorem 4.21]{Clarke3}
$$
f (g_0) = \lim_{k \rightarrow \infty} f (g_k). 
$$

Using this result, 
Clarke and Rubinstein (\cite{CR2})
showed that $d_{g_E/f^p}$ is a metric for any $p \in \mathbb{Z}$ and 
determined the metric completion $\overline{(\Mm, d_{g_E/f^p})}$ 
of $\Mm$ w.r.t. $d_{g_E/f^p}$.

\begin{theorem}[{\cite[Theorem 5.3]{CR2}}] \label{thm:CR}
The metric completion $\overline{(\Mm, d_{g_E/f^p})}$ 
of $\Mm$ w.r.t. $d_{g_E/f^p}$ is identified with the following. 
\begin{enumerate}
\renewcommand{\labelenumi}{(\arabic{enumi})}
\item If $p=1$, 
$\widehat{\Mm_{finite, +}} := \Mm_{finite, +}/\sim$, where 
$\Mm_{finite, +} = \{ g \in \Mm_{finite} \mid f(g)>0 \}$. 
\item If $p<1$, 
$\widehat{\Mm_{finite}}$. 
\item 
If $p >1$, 
$\widehat{\Mm_{finite, +}} \cup \{ g_\infty \}$, 
where $g_\infty$ corresponds to the single equivalence class of Cauchy
sequences $\{ h_k \}$ with $\lim_{k \rightarrow \infty} f(h_k) = \infty$. 
\end{enumerate}
\end{theorem}

Now we show that we can generalize Theorem \ref{thm:CR} 
by our method. First, we prove the following.

\begin{proposition} \label{prop:homog pair riem}
The pair $(g_E, f)$ is a homogeneous pair of degree $n/2$ 
w.r.t. the canonical $\R_{>0}$-action on $\Mm$. 
\end{proposition}
\begin{proof}
By the definitions of $g_E$ and $f$, 
we see that 
(\ref{eq:homog g}) and (\ref{eq:homog f}) are satisfied for $\alpha = n/2$. 
The vector field $P$ generated by 
the canonical $\R_{>0}$-action on $\Mm_{{\rm Spin}(7)}$ 
is given by 
$P_g = g$ at $g \in \Mm$. 
Then for any $h \in T_g \Mm \cong \Gamma (S^2 T^* M)$ we compute 
\begin{align*}
(df)_g (h) 
=  \int_M {\rm tr} (g^{-1} h) {\rm vol}_g
= (g_E)_g (P_g, h), 
\end{align*}
and hence (\ref{eq:thm assump}) is satisfied. 
\end{proof}

Then by Theorems \ref{thm:comp}, \ref{thm:Clarke3} and 
Proposition \ref{prop:comp level}, we obtain the following.

\begin{theorem} \label{thm:comp riem}
Use the notation of Theorem \ref{thm:CR}. 
Let $v: \R_{>0} \rightarrow \R_{>0}$ be a smooth function.  
Let $\hat T_0$ and $\hat T_\infty$ be defined in (\ref{eq:tilde T0 Tinf}). 
Then the metric completion $\overline{(\Mm, d_{(v \circ f) g_E})}$ w.r.t. $(v \circ f) g_E$ 
is identified with the following. 

\begin{enumerate}
\renewcommand{\labelenumi}{(\arabic{enumi})} 
\item 
If $\hat T_0 = - \infty$ and $\hat T_\infty = \infty$, 
$$
\widehat{\Mm_{finite, +}}. 
$$ 
\item 
If $\hat T_0 \in \R$, $\hat T_\infty = \infty$ and $\lim_{r \rightarrow 0} r v(r) =0$, 
$$
\widehat{\Mm_{finite}}. 
$$
\item 
If $\hat T_0 = - \infty$, $\hat T_\infty \in \R$ and $\lim_{r \rightarrow \infty} r v(r) =0$,  
$$
\widehat{\Mm_{finite, +}} \cup \{ g_\infty \}. 
$$
\item 
If $\hat T_0 \in \R$, $\hat T_\infty \in \R$, $\lim_{r \rightarrow 0} r v(r) =0$ and 
$\lim_{r \rightarrow \infty} r v(r) =0$,  
$$
\widehat{\Mm_{finite}} \cup \{ g_\infty \}. 
$$
\end{enumerate}
\end{theorem}

\begin{proof}
By Proposition \ref{prop:comp level}, 
the metric completion $\overline{\Mm_1}$ of $\Mm_1 = f^{-1}(1) \subset \Mm$ 
w.r.t. the metric induced from the induced Riemannian metric from $g_E$ is homeomorphic to 
$$
\left\{ [ g_k ] \in \overline{(\Mm, d_{g_E})} 
\ \middle| \ \lim_{k \rightarrow \infty} f(g_k) = 1 \right \}. 
$$
By the proof of Theorem \ref{thm:Clarke3}, 
this is identified with 
$\widehat{\Mm_{finite, 1}} := \Mm_{finite, 1}/ \sim$, 
where $\Mm_{finite, 1} = \{ g \in \Mm_{finite} \mid f(g)=1 \}.$ 
Then since there are canonical bijections between 
$\R_{>0} \times \overline{\Mm_1}$, 
$(\{ 0 \} \cup \R_{> 0}) \times \overline{\Mm_1} / \left( \{ 0 \} \times \overline{\Mm_1} \right)$, 
$(\R_{> 0} \cup \{ \infty \} ) \times \overline{\Mm_1} / 
\left( \{ \infty \} \times \overline{\Mm_1} \right)$, 
$
\left(\{ 0\} \cup \R_{> 0} \cup \{ \infty \} \right) \times \overline{\Mm_1} / 
\left( \{0, \infty \} \times \overline{\Mm_1} \right)
$
and 
$\widehat{\Mm_{finite, +}}$, $\widehat{\Mm_{finite}}$, $\widehat{\Mm_{finite, +}} \cup \{ g_\infty \}$, 
$\widehat{\Mm_{finite}} \cup \{ g_\infty \}$, 
respectively, 
the proof is completed by Theorem \ref{thm:comp}. 
\end{proof}

\begin{remark} 
This theorem generalizes Theorem \ref{thm:CR}. 
The weak Riemannian metric $(v \circ f) g_E$ was first considered 
by Bauer, Harms and Michor in \cite{BHM}. 
They also considered weak Riemannian metrics weighted by scalar curvature 
and described the geodesic equation for these weak Riemannian metrics. 
Then they showed that the exponential mapping for some of them is a local diffeomorphism. 

The weak Riemannian metric $g_E/f$ was the first example 
whose metric completion is strictly smaller than that of the Ebin metric $g_E$. 
We can give infinitely many such examples by Theorem \ref{thm:comp riem} (1). 
\end{remark}


\subsubsection{The Teichm\"uller space} \label{subsec:Teich}

In addition to the setting of Section \ref{subsec:Riem met}, 
suppose that $M$ is a compact Riemann surface of genus $\k \geq 2$. 
Let $\Mm_{<0} \subset \Mm$ be the space of all 
smooth Riemannian metrics of constant negative sectional curvature on $M$.

The restrictions of $g_E$ and $f$ in (\ref{eq:Ebin metric}) and (\ref{eq:Ebin vol}) to $\Mm_{<0}$ 
define a Riemannian-metric and a function on $\Mm_{<0}$. 
These are invariant under the action of 
${\rm Diff}_0 (M)$, the identity component of the diffeomorphism group. 
Thus they induce 
a Riemannian metric and a function on $\Tt_{<0} := \Mm_{<0}/{\rm Diff}_0 (M)$. 
By an abuse of notation, we denote these by $g_E$ and $f$. 
Proposition \ref{prop:homog pair riem} implies that 
$(g_E, f)$ is a homogeneous pair of degree $1$ 
w.r.t. the canonical $\R_{>0}$-action on $\Mm_{<0}$.

By the Gauss-Bonnet formula, we have 
$$
K_g \cdot f(g) = 2\int_M K_g {\rm vol}_g = 4 \pi (2-2 \k). 
$$ 
for $g \in \Mm_{<0}$. 
Thus setting $l=8 \pi (\k -1)$, we have 
$$
\Tt_l := \{ [g] \in \Tt_{<0} \mid f([g]) = l \} = \{ [g] \in \Tt_{<0} \mid K_g =-1 \},  
$$
which is called the Teichm\"uller space of $M$. 
The induced Riemannian metric on $\Tt_l$ from $g_E$ is called the Weil-Petersson metric. 

This space is well understood. 
The space $\Tt_l$ is known to be a $(6 \k - 6)$-dimensional manifold homeomorphic to $\R^{6 \k -6}$. 
Since $g_E$ and $f$ are invariant under the action of ${\rm Diff}_+ (M)$, 
where ${\rm Diff}_+ (M)$ is the group of orientation preserving diffeomorphisms of $M$, 
they induce a Riemannian metric and a function on 
the orbifold $\Mm_{<0}/{\rm Diff}_+ (M) = \Tt_{<0}/{\rm MCG}(M)$, 
where 
${\rm MCG}(M)={\rm Diff}_+ (M)/{\rm Diff}_0 (M)$ is the mapping class group. 
By an abuse of notation, we denote these by $g_E$ and $f$. 
Then the metric completion of $\Tt_l/{\rm MCG}(M)$ 
w.r.t. the metric induced from $g_E$ is homeomorphic to the Deligne-Mumford compactification 
of the moduli space of Riemann surfaces of genus $\k$, 
which is a projective algebraic variety.

The statements in this paper would be true for orbifolds. 
On the orbifold 
$
\Mm_{<0}/{\rm Diff}_+ (M) = \Tt_{<0}/{\rm MCG}(M), 
$
$(g_E, f)$ is a homogeneous pair of degree $1$ 
w.r.t. the canonical $\R_{>0}$-action on $\Mm_{<0}/{\rm Diff}_+ (M)$ 
by Proposition \ref{prop:homog pair riem}. 
Then we have the metric completion as in Theorem \ref{thm:comp}. 
In particular, for a function $v:\R_{>0} \rightarrow \R_{>0}$ 
corresponding to the case (4) in Theorem \ref{thm:comp}, 
the metric completion of $\Mm_{<0}/{\rm Diff}_+ (M)$
w.r.t. the metric induced from $(v \circ f) g_E$ will be compact 
by Remark \ref{rem:comp rough}. 
It will be interesting if we can know that 
the metric completion of $\Mm_{<0}/{\rm Diff}_+ (M)$ is also 
a projective algebraic variety for some $v$.

\appendix
\section{Appendix} \label{app:notation}

We summarize the notations and basic definitions used in this paper. 

\begin{definition} \label{def:def}
Let $(M, g)$ be a pseudo-Riemannian manifold. 
We call a pseudo-Riemannian metric {\bf definite} 
if it is positive or negative definite. 
\end{definition}

\begin{definition} \label{def:comp}
Let $(Z,d)$ be a metric space. 
The {\bf metric completion} $\overline{Z}$ w.r.t. the metric $d$ is defined by 
$\overline{Z} = Z_C/\sim$, 
where $Z_C$ is the space of Cauchy sequences in $Z$ 
and $\sim$ is the equivalence relation defined by 
$
\{ z_k\} \sim \{ z'_k \} \Leftrightarrow \lim_{k \rightarrow \infty} d (z_k, z'_k) =0. 
$
Denote by $[z_k]$ the equivalence class of $\{ z_k \}$. 
Then $\overline{Z}$ is a metric space with the metric 
$
d ([z_k], [z'_k]) = \lim_{k \rightarrow \infty} d (z_k, z'_k), 
$
where we also use $d$ to describe the metric on $\overline{Z}$ by an abuse of notation. 
\end{definition}

We summarize the notations used in this paper. 
In the following table, $(M, g)$ is a pseudo-Riemannian manifold.

\vspace{0.5cm}
\hspace{-1cm}
\begin{tabular}{|lc|l|}
\hline
Notation                        & & Meaning \\ \hline \hline
$\R_{>0}$                       &  & $\R_{>0} = \{ x \in \R \mid x>0 \}$\\
 $i(\cdot)$                     &  & The interior product \\
$|v|_g$                           &  & $|v|_g = \sqrt{g(v,v)}$ for $v \in TM$ when $g$ is positive-definite\\
$d_g$                            &  & The induced (pseudo)metric from $g$ when $g$ is positive-definite\\
$\overline{M}$                & & The metric completion of $M$ w.r.t. $d_g$ when $g$ is positive-definite\\
$[x_k]$                          &  & The equivalence class in $\overline{M}$ of a Cauchy sequence 
                                           $\{ x_k \} \subset M$\\
${\rm grad}^g f$              & & The gradient vector field of a function $f$ defined by 
$g({\rm grad}^g f, \cdot) = df$\\
$v^{\flat} \in T^* M$       & & $v^{\flat} = g(v, \cdot)$ for $v \in TM$ \\ 
$\alpha^{\sharp} \in TM$ & & $\alpha = g(\alpha^{\sharp}, \cdot)$ for $\alpha \in T^* M$ \\
$\nabla^g$                     & & The Levi-Civita connection of $g$\\
$R^g$                            & & The curvature tensor: $R^g (A,B) = [\nabla^g_A, \nabla^g_B] 
- \nabla^g_{[A,B]}$ for $A,B \in TM$\\
$K^g$                            & & The sectional curvature of $g$\\
$\mathfrak{X}(M)$           & & The space of smooth vector fields on $M$\\
$\r', \r'', \cdots$             & & For a function of one variable $\rho = \rho(r)$, 
$\r' = d \r/d r, \r''=d^2 \r/d r^2, \cdots$\\
$\dot r, \ddot r, \cdots$  & &  For a function of one variable $r = r (t)$, $\dot r = d r/d t, \ddot r = d^2 r/d t^2, \cdots$\\
$\p_r$                           & & $\p_r = \p/\p r$\\
$K^g(\p_r)$                    & & $K^g(\p_r) = \frac{- 2 \r'' \xi + \r' \xi'}{2 \r \xi^2}$ defined in (\ref{eq:def Kpr})\\
$E_1$                            & & $E_1=g_Y (\dot y_0, \dot y_0) (\r (r_0))^4$ 
defined in Proposition \ref{prop:geod re} \\
$E_2$                            & & $E_2= \xi (r_0)  (\dot r_0)^2 + g_Y (\dot y_0, \dot y_0) (\r(r_0))^2$ 
defined in Lemma \ref{lem:wp geod 2} \\
$E_3$                            & & $E_3 = \r (r_0)^2$ defined in Proposition \ref{prop:geod re} \\
$F$                               & & $F= \frac{1}{4} \left(\left( \frac{\dot r_0}{r_0} \right)^2 + \frac{g_Y (\dot y_0, \dot y_0)}{k} \right)$ defined in  Proposition \ref{prop:geod wp explicit} \\
$g(w)$                           & & $g(w) = \frac{k w(r)}{r^2} d r^2+ w(r) g_Y$ given in (\ref{eq:def gw})\\
\hline
\end{tabular}

\address{
Gakushuin University, 1-5-1, Mejiro, Toshima,Tokyo, 171-8588, Japan}
{kkawai@math.gakushuin.ac.jp}

\end{document}